\documentclass[a4paper]{article}

\usepackage{amsthm}
\usepackage{amssymb}
\usepackage{amsmath}

\newtheorem{theorem}{Theorem}[section]
\newtheorem{conjecture}[theorem]{Conjecture}
\newtheorem{lemma}[theorem]{Lemma}
\newtheorem*{definition}{Definition}

\renewenvironment{proof}[1][]%
{\noindent {\setcounter{equation}{0}\it Proof.
}{#1}{}}{\hfill$\Box$\vspace{2ex}}

\def\longbox#1{\parbox{0.85\textwidth}{#1}}

\title{On the choosability of claw-free perfect graphs}
\author{Sylvain Gravier\thanks{CNRS, Institut Fourier, University of
Grenoble,
France.} \and Fr\'ed\'eric
Maffray\thanks{CNRS, Laboratoire G-SCOP, University of Grenoble,
France.} \and Lucas Pastor\thanks{Laboratoire G-SCOP, University of
Grenoble,
France.}}

\begin{document}

\maketitle

\begin{abstract}
It has been conjectured that for every claw-free graph $G$ the choice
number of $G$ is equal to its chromatic number.  We focus on the
special case of this conjecture where $G$ is perfect.  Claw-free
perfect graphs can be decomposed via clique-cutset into two special
classes called elementary graphs and peculiar graphs.  Based on this
decomposition we prove that the conjecture holds true for every
claw-free perfect graph with maximum clique size at most $4$.
\end{abstract}


\section{Introduction}

We consider finite, undirected graphs, without loops.  Given a graph
$G$ and an integer $k$, a \emph{$k$-coloring} of the vertices of $G$
is a mapping $c : V(G) \rightarrow \{1,2, \ldots, k\}$ for which every
pair of adjacent vertices $x,y$ satisfies $c(x) \neq c(y)$.  A
\emph{coloring} is a $k$-coloring for any $k$.  The graph $G$ is
called \emph{$k$-colorable} if it admits a $k$-coloring.  The
\emph{chromatic number} of $G$, denoted by $\chi(G)$, is the smallest
integer $k$ such that $G$ is $k$-colorable.

\medskip

The \emph{list-coloring} variant of the coloring problem, introduced
by Erd\H{o}s, Rubin and Taylor \cite{ErRuTa79} and by Vizing
\cite{JGT:JGT3190160510}, is as follows.  Assume that each vertex $v$
has a list $L(v)$ of prescribed colors; then we want to find a
coloring $c$ such that $c(v) \in L(v)$ for all $v \in V(G)$.  When
such a coloring exists we say that the graph $G$ is
\emph{$L$-colorable} and that $c$ is an \emph{$L$-coloring} of $G$.
Given an integer $k$, a graph $G$ is \emph{$k$-choosable} if it is
$L$-colorable for every assignment $L$ that satisfies $|L(v)| = k$ for
all $v \in V(G)$ (equivalently, if it is $L$-colorable for every
assignment $L$ that satisfies $|L(v)| = k$ for all $v \in V(G)$).  The
\emph{choice number} or \emph{list-chromatic number} $ch(G)$ of $G$ is
the smallest $k$ such that $G$ is $k$-choosable.  It is easy to see
that every $k$-choosable graph $G$ is $k$-colorable (consider the
assignment $L(v) = \{1,2, \ldots, k\}$ for all $v \in V(G)$), and so
$\chi(G) \leq ch(G)$ holds for every graph.  There are graphs for
which the difference between $ch(G)$ and $\chi(G)$ is arbitrarily
large.  (For example, it is easy to see that the choice number of the
complete bipartite graph $K_{p,p^p}$ is $p+1$.)

\medskip

The above notions can be extended to the problem of coloring the edges
of a graph.  The least number of colors necessary to color all edges
of a graph in such a way that no two adjacent edges receive the same
color is its \emph{chromatic index} $\chi'(G)$.  The least $k$ such
that $G$ is $L'$-edge-colorable for any assignment $L'$ of colors to
the edges of $G$ with $|L'(e)|=k$ for all $e\in E$ is called the
\emph{choice index} or \emph{list-chromatic index} of $G$.  Vizing
(see \cite{JGT:JGT3190160510}), proposed the following conjecture:
\begin{conjecture}\label{conj1}
Every graph $G$ satisfies $ch'(G)=\chi'(G)$.
\end{conjecture}
The special case of this conjecture dealing with list-coloring the
edges of a complete bipartite graph was known as the Dinitz
conjecture, as it was equivalent to a problem on Latin squares posed
by Jeffrey Dinitz.  Galvin \cite{Galvin1995153} established the
following more general result.
\begin{theorem}[Galvin \cite{Galvin1995153}]
Every bipartite graph $G$ satisfies $ch'(G)=\chi'(G)$.
\end{theorem}

The problem of edge-coloring can be reduced to a special instance of
the problem of vertex-coloring via the line-graph.  Given a graph $H$,
the \emph{line-graph} ${\cal L}(H)$ of $H$ is the graph whose vertices
are the edges of $H$ and whose edges are the pairs of adjacent edges
of $H$.  Conversely, $H$ is called the \emph{root} graph of ${\cal L}
(H)$.  It is clear that $\chi({\cal L}(H)) =\chi'(H)$ and $ch({\cal
L}(H))=ch'(H)$.

\medskip

In a graph $G$, we say that a vertex $v$ is \emph{complete} to a set
$S\subseteq V(G)$ when $v$ is adjacent to every vertex in $S$, and
\emph{anticomplete} to $S$ when $v$ has no neighbor in $S$.  Given two
sets $S,T\subseteq V(G)$ we say that $S$ is \emph{complete} to $T$ is
every vertex in $S$ is adjacent to every vertex in $T$, and
\emph{anticomplete} to $T$ when no vertex in $S$ is adjacent to any
vertex in $T$.  The neighborhood of a vertex $v$ is denoted by
$N_G(v)$ (and the subscript $G$ may be dropped when there is no
ambiguity).  The complement of graph $G$ is denoted by $\overline{G}$.

A graph is \emph{cobipartite} if its complement is bipartite, in other
words if its vertex-set can be partitioned into at most two cliques.
We let $P_n$, $C_n$ and $K_n$ respectively denote the path, cycle and
complete graph on $n$ vertices.

Given any graph $F$, a graph $G$ is \emph{$F$-free} if no induced
subgraph of $G$ is isomorphic to $F$.  The \emph{claw} is the graph
with four vertices $a,b,c,d$ and edges $ab$, $ac$, $ad$; vertex $a$ is
called the \emph{center} of the claw.

A graph $G$ is \emph{perfect} if every induced subgraph $H$ of $G$
satisfies $\chi(H) = \omega(H)$.  A \emph{Berge} graph is any graph
that does not contain as an induced subgraph an odd cycle of length at
least five or the complement of an odd cycle of length at least five.
Chudnovsky, Robertson, Seymour, Thomas solved the long-standing and
famous problem known as the Strong Perfect Graph Conjecture by proving
the following theorem.
\begin{theorem}[\cite{Chudnovsky2006}]
A graph $G$ is perfect if and only if it is Berge.
\end{theorem}

The special case of the Strong Perfect Graph Conjecture concerning
claw-free graphs had been resolved much earlier by Parthasarathy and
Ravindra.
\begin{theorem}[Parthasarathy and Ravindra~\cite{Parthasarathy1976}]
Every claw-free Berge graph $G$ is perfect.
\end{theorem}

Here we are interested in a restricted version of a question posed by
two of us
\cite{Gravier:1997:CNE:249460.249532,Gravier:1998:GWC:1378752.1378756},
asking whether it is true that every claw-free graph $G$ satisfies
$ch(G) = \chi(G)$.
\begin{conjecture}
Every claw-free perfect graph $G$ satisfies $ch(G) = \chi(G)$.
\end{conjecture}
This conjecture was proved in \cite{Gravier2004211} for every
claw-free perfect graph $G$ with $\omega(G) \leq 3$.  Here we will
prove it for the case $\omega(G) \leq 4$.  Our main result is the
following.
\begin{theorem}\label{theorem:claw-free-perfect}
Let $G$ be a claw-free perfect graph with $\omega(G) \leq 4$. Then
$ch(G) = \chi(G)$.
\end{theorem}

Our proof is based on a decomposition theorem for claw-free perfect
graphs due to Chv\'atal and Sbihi \cite{Chvatal1988154}.  They proved
that every claw-free perfect graph either admits a clique cutset or
belongs to two specific classes of graphs, which we defined precisely
below.

\begin{definition}[Clique cutset]
A \emph{clique cutset} in a graph $G$ is a clique $C$ of $G$ such that
$G \setminus C$ is disconnected. A \emph{minimal clique cutset} is a
clique cutset that does not contain another clique cutset.
\end{definition}
If $C$ is a minimal clique cutset in a graph $G$ and $A_1, \ldots,
A_k$ are the vertex-sets of the components of $G \setminus C$, we
consider that $G$ is decomposed into the collection of induced
subgraphs $G[A_i \cup C]$ for $i = 1, \ldots, k$.  These subgraphs
themselves may admit clique cutsets, so the decomposition (via minimal
clique cutsets) can be applied further.  This decomposition can be
represented by a tree, where each non-leaf node corresponds to an
induced subgraph $G'$ of $G$ and a minimal clique cutset $C'$ of $G'$,
and the children of the node are the induced subgraphs into which $G'$
is decomposed along $C'$.  The leaves of $T$ are indecomposable
subgraphs of $G$ (subgraphs that have no clique cutset), which we call
\emph{atoms}.  (This tree may not be unique, depending on the choice
of a clique cutset at each node.)  Whitesides \cite{whi} and Tarjan
\cite{tar} proved that for every graph $G$ on $n$ vertices every
clique-cutset decomposition tree has at most $n$ leaves and that such
a decomposition can be obtained in polynomial time $O(n^3)$.  A nice
feature is that every graph $G$ admits an \emph{extremal} clique
cutset, that is, a minimal clique cutset $C$ such that there is a
component $H$ of $G\setminus C$ such that $G[V(H)\cup C]$ is an atom.

\begin{definition}[Elementary graph \cite{Chvatal1988154}]
A graph is \emph{elementary} if its edges can be colored with two
colors (one color on each edge) in such a way that every induced
two-edge path has its two edges colored differently.
\end{definition}

\begin{definition}[Peculiar graph \cite{Chvatal1988154}]
A graph $G$ is \emph{peculiar} if $V(G)$ can be partitioned into nine
sets $A_i, B_i, Q_i$ ($i=1,2,3$) that satisfy the following
properties for each $i$, where subsbcripts are understood modulo~$3$:
\begin{itemize}
\item
Each of the nine sets is non-empty and induces a clique.
\item
$A_i$ is complete to $B_i\cup A_{i+1}\cup A_{i+2} \cup B_{i+2}$ and
not complete to $B_{i+1}$.
\item
$B_i$ is complete to $A_i\cup B_{i+1}\cup B_{i+2} \cup A_{i+1}$ and
not complete to $A_{i+2}$.
\item
$Q_i$ is complete to $A_{i+1}\cup B_{i+1}\cup A_{i+2}\cup B_{i+2}$ and
anticomplete to $A_i\cup B_i\cup Q_{i+1}\cup Q_{i+2}$.
\end{itemize}
We say that $(A_1, B_1, A_2, B_2, A_3, B_3, Q_1, Q_2, Q_3)$ is a
peculiar partition of $G$.
\end{definition}

\begin{theorem}[Chv\'atal and Sbihi \cite{Chvatal1988154}]
\label{thm:chvsbi}
Every claw-free perfect graph either has a clique cutset or is a
peculiar graph or an elementary graph.
\end{theorem}

The structure of peculiar graphs is clear from their definition.
Concerning elementary graphs, their structure was elucidated by
Maffray and Reed \cite{Maffray1999134} as follows.  Let us say that an
edge is \emph{flat} if it is not contained in a triangle.

\begin{definition}[Flat edge augmentation]
Let $xy$ be a flat edge in a graph $G$, and let $A$ be a cobipartite
graph such that $V(A)$ is disjoint from $V(G)$ and $V(A)$ can be
partitioned into two cliques $X,Y$.  We obtain a new graph $G'$ by
removing $x$ and $y$ from $G$ and adding all edges between $X$ and
$N_G(x) \setminus \{y\}$ and all edges between $Y$ and $N_G(y)
\setminus \{x\}$.  This operation is called \emph{augmenting} the flat
edge $xy$ with the cobipartite graph $A$.  In $G'$ the pair $(X,Y)$ is
called the \emph{augment}.
\end{definition}
When $x_1y_1, \ldots, x_ky_k$ are pairwise non-adjacent flat edges in
a graph $G$, and $A_1, \ldots, A_k$ are pairwise vertex-disjoint
cobipartite graphs, also vertex-disjoint from $G$, one can augment
each edge $x_iy_i$ with the graph $A_i$.  Clearly the result is the
same whatever the order in which the $k$ operations are performed.  We
say that the resulting graph is an \emph{augmentation} of $G$.

\begin{theorem}[Maffray and Reed \cite{Maffray1999134}]
\label{thm:mafree}
A graph $G$ is elementary if and only if it is an augmentation of the
line-graph $H$ of a bipartite multigraph $B$.  Moreover we may assume
that each augment $A_i$ satisfies the following:
\begin{itemize}
\item
There is at least one pair of non-adjacent vertices in $A_i$,
\item
The bipartite graph whose vertex-set is $X_i\cup Y_i$ and whose edges
are the edges of $A_i$ with one end in $X_i$ and one in $Y_i$ is
connected (and consequently both $|X_i|,|Y_i|\ge 2$).
\end{itemize}
\end{theorem}

In a directed graph $D$, for every vertex $v$ we let $d^+(v)$ denote
the number of vertices $w$ such that $vw$ is an arc of $D$.
\begin{theorem}[Galvin \cite{Galvin1995153}]\label{thm:galvin}
Let $G$ be the line-graph of a bipartite graph $B$, where $V(B)$ is
partitioned into two stable set $X,Y$.  Let $f$ be an
$\omega(G)$-coloring of the vertices of $G$, with colors
$1,2,\ldots,\omega(G)$.  Let $D$ be the directed graph obtained from
$G$ by directing every edge $uv$ as follows, assuming that
$f(u)<f(v)$: when the common end of edges $u,v$ in $B$ is in $X$, then
give the orientation $u\rightarrow v$, and when it is in $Y$ give the
orientation $u\leftarrow v$.  Assume that $L$ is a list assignment on
$V(G)$ such that every vertex $v$ of $G$ satisfies $|L(v)|\ge
d_D^+(v)+1$.  Then $G$ is $L$-colorable.
\end{theorem}

Let $G$ be a graph and let $L$ be a list assignment on $V(G)$. For
every set $S\subseteq V(G)$ we set $L(S) = \bigcup_{x \in S} L(x)$.
If $f$ is a coloring of $G$, we set $f(S)=\{f(x)\mid x\in S\}$. If
$H$ is an induced subgraph of $G$, we may also write $L(H)$ and $f(H)$
instead of $L(V(H))$ and $f(V(H))$ respectively.

\medskip

For the sake of completeness we recall the classical theorems of
K\H{o}nig and Hall.  Let $X_1, \ldots, X_k$ be a family of sets.  A
\emph{system of distinct representatives} for the family is a subset
$\{x_1, \ldots, x_k\}$ of $k$ distinct elements of $X_1\cup\cdots\cup
X_k$ such that $x_i\in X_i$ for all $i=1, \ldots, k$.  Note that if
$G$ is a graph and $L$ is a list assignment on $V(G)$, and the family
$\{L(v)\mid v\in V(G)\}$ admits a system of distinct representatives,
then this is an $L$-coloring of $G$.
\begin{theorem}[Hall's theorem \cite{Hall,Sch}]
A family $\cal F$ of $k$ sets has a system of distinct representatives
if and only if, for all $\ell\in\{1,\ldots,k\}$, the union of any
$\ell$ members of $\cal F$ has size at least $\ell$.
\end{theorem}

A \emph{matching} in a graph $G$ is a set of pairwise non-incident
edges.
\begin{theorem}[K\H{o}nig's theorem \cite{Sch}]
In a bipartite graph on $n$ vertices, let $\mu$ be the size of a
maximum matching and $\alpha$ be the size of a maximum stable set.
Then $\mu+\alpha=n$.
\end{theorem}

\section{Peculiar graphs}

\begin{lemma}\label{lem:peculiar}
Let $G$ be a connected claw-free graph that contains a peculiar
subgraph, and assume that $G$ is also $C_5$-free.  Then $G$ is
peculiar.
\end{lemma}
\begin{proof}
Let $H$ be a peculiar subgraph of $G$ that is maximal.  If $H=G$ we
are done.  So let us assume that $H\neq G$.  Since $G$ is connected
there is a vertex $x$ of $V(G)\setminus V(H)$ that has a neighbor in
$H$.  Let $A_1, B_1, A_2, B_2, A_3, B_3, Q_1, Q_2, Q_3$ be nine
cliques that form a partition of $V(H)$ as in the definition of a
peculiar graph.  For $i=1, 2, 3$ we pick a pair of non-adjacent
vertices $a_i\in A_i$ and $b_{i+1}\in B_{i+1}$, and we pick any
$q_i\in Q_i$.  (All subscripts are modulo $3$.)

If $x$ has no neighbor in $Q_1\cup Q_2\cup Q_3$, then it has a
neighbor $a$ in $A_i\cup B_i$ for some $i$; but then $\{a, x,
q_{i+1}, q_{i+2}\}$ induces a claw. Therefore $x$ has a neighbor in
$Q_1\cup Q_2\cup Q_3$.

Suppose that $x$ has a neighbor $k$ in $Q_1$ and none in $Q_2\cup
Q_3$.  Then $x$ has no neighbor $z$ in $A_1\cup B_1$, for otherwise
$\{z, x, q_2, q_3\}$ induces a claw.  Also $x$ is adjacent to one of
$a_2, b_3$, for otherwise $\{x, k, a_2, b_3\}$ induces a claw; up to
symmetry we assume that $x$ is adjacent to $a_2$.  Then $x$ is
adjacent to every vertex $a\in A_3$, for otherwise $\{a_2, q_3, a,
x\}$ induces a claw; and to every vertex $y\in A_2\cup B_2\cup Q_1$,
for otherwise $\{a_3, y, x, q_2\}$ induces a claw; and to every vertex
$b\in B_3$, for otherwise $\{b_2, b, q_3, x\}$ induces a claw.  Hence
$x$ is complete to $A_2\cup B_2\cup A_3\cup B_3\cup Q_1$ and
anticomplete to $A_1\cup B_1\cup Q_2\cup Q_3$.  So $V(H)\cup\{x\}$
induces a peculiar subgraph of $G$, because $x$ can be added to $Q_1$,
a contradiction to the choice of $H$.

Therefore we may assume up to symmetry that $x$ has a neighbor $k\in
Q_1$ and a neighbor $k'\in Q_2$.  Note that $x$ has no neighbor
$k''\in Q_3$, for otherwise $\{x, k, k', k''\}$ induces a claw.

Suppose that $x$ has a non-neighbor $a\in A_1$.  Then $x$ is adjacent
to every vertex $u\in A_2$, for otherwise $\{x, k, u, a, k'\}$ induces
a $C_5$; and then to every vertex $v\in B_2$, for otherwise either
$\{a_2, a, x, v\}$ induces a claw (if $av\notin E(G)$) or $\{x, k, v,
a, k'\}$ induces a $C_5$ (if $av\in E(G)$); and then to every vertex
$w\in A_3\cup B_3\cup Q_1$, for otherwise $\{b_2,x,w,q_3\}$ induces a
claw.  Then $a$ is adjacent to every vertex $b\in B_2$, for otherwise
$\{x,k',a,q_3,b\}$ induces a $C_5$; and by the same argument the set
$A_1\setminus N(x)$ is complete to $B_2$.  It follows that $a_1\in
N(x)$ since $a_1$ is not complete to $B_2$.  Then $x$ is adjacent to
every vertex $q\in Q_2$, for otherwise $\{a_1, x, q_3, q\}$ induces a
claw.  But now we observe that $V(H)\cup\{x\}$ induces a larger
peculiar subgraph of $G$, because $x$ can be added to $A_3$ and the
vertices of $A_1\setminus N(x)$ can be moved to $B_1$.

Therefore we may assume that $x$ is complete to $A_1$, and, similarly,
to $B_2$.  Then $x$ is adjacent to every vertex $u$ in $Q_2\cup B_3$,
for otherwise $\{a_1, x, u, q_3\}$ induces a claw, and similarly $x$
is complete to $Q_1\cup A_3$.  It cannot be that $x$ has both a
non-neighbor $a'\in A_2$ and a non-neighbor $b'\in B_1$, for otherwise
$\{x,k,a',b',k'\}$ induces a $C_5$.  So, up to symmetry, $x$ is
complete to $A_2$.  But now $V(H)\cup\{x\}$ induces a larger peculiar
subgraph of $G$, because $x$ can be added to $A_3$.  This completes
the proof of the lemma.
\end{proof}

We observe that (up to isomorphism) there is a unique peculiar graph
$G$ with $\omega(G)=4$.  Indeed if $G$ is such a graph, with the same
notation as in the definition of a peculiar graph, then for each $i$
the set $Q_i\cup A_{i+1} \cup B_{i+1}\cup A_{i+2}$ is a clique, so,
since $G$ has no clique of size~$5$, the four sets $Q_i, A_{i+1},
B_{i+1}, A_{i+2}$ have size~$1$; and so the nine sets $A_i, B_i, Q_i$
($i=1,2,3$) all have size~$1$.  Hence $G$ is the unique peculiar graph
on nine vertices.

\begin{lemma}\label{lem:colpec4}
Let $G$ be a peculiar graph with $\omega(G)=4$.  Then $G$ is
$4$-choosable.
\end{lemma}
\begin{proof}
Let $(A_1, B_1, A_2, B_2, A_3, B_3, Q_1, Q_2, Q_3)$ be a peculiar
partition of $G$.  As observed above, we have $|A_i|=|B_i|=|Q_i|=1$
for all $i=1,2,3$.  Hence let $A_i=\{a_i\}$, $B_i=\{b_i\}$ and
$Q_i=\{q_i\}$, for all $i=1,2,3$.  Recall that $a_i$ is not adjacent
to $b_{i+1}$, for each $i$.  Let $Q=\{q_1,q_2,q_3\}$.

Let $L$ be a list assignment that satisfies $|L(v)|= 4$ for all $v\in
V(G)$.  Let us prove that $G$ is $L$-colorable.

First suppose that for some $i\in\{1,2,3\}$ we have $L(a_i) \cap
L(b_{i+1}) \neq \emptyset$, say for $i=1$.  Pick any $c\in L(a_1) \cap
L(b_{2})$.  Let $G'=G\setminus\{a_1,b_{2}\}$ and let
$L'(x)=L(x)\setminus\{c\}$ for all $x\in V(G')$.  Clearly, $G'$ is a
claw-free perfect graph and $\omega(G') = 3$.  Moreover, $G'$ is
elementary.  To see this, define an egde coloring of $G'$ by coloring
blue the edges in $\{q_3b_1, q_3a_2, b_1a_2, b_3a_3, q_2a_3, b_3q_1\}$
and red the edges in $\{q_2b_1, q_2b_3, b_3b_1, q_1a_2, q_1a_3,
a_2a_3\}$; it is a routine matter to check that this edge coloring is
an elementary coloring.  By \cite{Gravier2004211}, $G'$ is
$3$-choosable, so it admits an $L'$-coloring.  We can extend this
coloring to $a_1$ and $b_{2}$ by assigning color~$c$ to them.
Therefore we may assume that:
\begin{equation}\label{pec41}
L(a_i) \cap L(b_{i+1})=\emptyset \mbox{ for all } i=1,2,3.
\end{equation}

Now suppose that there are vertices $u,v\in Q$ such that $L(u) \cap
L(v) \neq \emptyset$.  Let $w$ be the unique vertex in $Q\setminus
\{u,v\}$.  Pick any $c\in L(u) \cap L(v)$.  Let
$G'=G\setminus\{u,v\}$.  Let $L'(x)=L(x)\setminus\{c\}$ for all $x\in
V(G')\setminus\{w\}$, and let $L'(w)=L(w)$.  We claim that the family
$\{L'(x)\mid x\in V(G')\}$ admits a system of distinct
representatives.  Suppose the contrary.  By Hall's theorem, there is a
set $S\subseteq V(G')$ such that $|L'(S)|<|S|$.  Since $|L'(x)|\ge 3$
for all $x\in V(G')$, we have $|L'(S)|\ge 3$, so $|S|\ge 4$; this
implies that either (a) $S\supseteq \{a_i, b_{i+1}\}$ for some
$i\in\{1,2,3\}$ or (b) $S$ contains $w$.  In case~(a), (\ref{pec41})
implies that $c$ belongs to at most one of $L(a_i)$ and $L(b_{i+1})$,
and so $|L'(S)|\ge |L'(a_i)\cup L'(b_{i+1})|\ge 7$, so $|S|\ge 8$,
which is impossible because $|V(G')|=7$.  In case~(b), since
$|L'(w)|=4$, we have $|L'(S)|\ge 4$, so $|S|\ge 5$, which implies that
$S$ satisfies (a) again, a contradiction.  Thus the family
$\{L'(x)\mid x\in V(G')\}$ admits a system of distinct
representatives, which is an $L'$-coloring of $G'$.  We can extend
this coloring to $u$ and $v$ by assigning color~$c$ to them.
Therefore we may assume that
\begin{equation}\label{pec42}
L(u) \cap L(v)=\emptyset \mbox{ for all } u,v\in Q.
\end{equation}

We claim that the family $\{L(x)\mid x\in V(G)\}$ admits a system of
distinct representatives.  Suppose the contrary.  By Hall's theorem,
there is a set $T\subseteq V(G)$ such that $|L(T)|<|T|$.  Since
$|L(x)|=4$ for all $x\in V(G)$, we have $|L(T)|\ge 4$, so $|T|\ge 5$;
this implies that either (a) $T\supseteq \{a_i, b_{i+1}\}$ for some
$i\in\{1,2,3\}$ or (b) $T$ contains two vertices from $Q$.  In either
case, (\ref{pec41}) or (\ref{pec42}) implies that $|L(T)|\ge 8$, so
$|T|\ge 9$, that is, $T=V(G)$.  But then $T\supset Q$, so
(\ref{pec42}) implies that $|L(T)|\ge 12$ and $|T|\ge 13$, which is
impossible.  Thus the family $\{L(x)\mid x\in V(G)\}$ admits a system
of distinct representatives, which is an $L$-coloring of $G$.
\end{proof}

\section{Cobipartite graphs}

In this section we analyze the list-colorability of certain
cobipartite graphs with certain list assignments.

\begin{lemma}\label{lemma:x-y-4}
Let $H$ be a cobipartite graph, where $V(H)$ is partitioned into two
cliques $X$ and $Y$.  Assume that $|X| \leq |Y|$ and that there are
$|X|$ non-edges between $X$ and $Y$ and they form a matching in
$\overline{H}$.  Let $L$ be a list assignment on $V(H)$ such that
$|L(x)| \ge |X|$ for all $x \in X$ and $|L(y)| \ge |Y|$ for all $y \in
Y$.  Then $H$ is $L$-colorable.
\end{lemma}
\begin{proof}
Let $X=\{x_1, \ldots, x_p\}$, and let $y_1,\ldots,y_p$ be vertices of
$Y$ such that $\{x_1,y_1\}$, \ldots, $\{x_p,y_p\}$ are the non-edges
of $H$.  The hypothesis implies that $y_1, \ldots, y_p$ are pairwise
distinct.  Since a clique in $H$ can contain at most one of $x_i,y_i$
for each $i=1,\ldots,p$, we have $\omega(H)=|Y|$.

We proceed by induction on $|X|$.  If $|X| = 0$, then $H$ is a clique
with $|L(v)| = |V(H)|$ for all $v \in V(H)$; so $H$ is $L$-colorable
by Hall's theorem.  Now suppose that $|X| > 0$.  If the family $\{L(v)
\mid v \in V(H)\}$ admits a system of distinct representatives, then
this is an $L$-coloring.  So suppose the contrary.  By Hall's theorem
there is a set $T \subseteq V(H)$ such that $|L(T)| < |T|$.  Then
$|T|> |X|$, so $T$ contains a vertex $y$ from $Y$, and so $|T|>
|L(y)|\ge |Y|$.  Since $\omega(H)=|Y|$, it follows that $T$ is not a
clique.  So $T$ contains non-adjacent vertices $x, y$ with $x \in X$
and $y \in Y$.  We have $|L(x) \cup L(y)| \leq |L(T)| < |T| \leq |X| +
|Y|$, which implies $L(x) \cap L(y) \neq \emptyset$.  Pick a color
$c\in L(x)\cap L(y)$.  Set $L'(w) = L(w) \setminus \{c\}$ for all $w
\in V(H) \setminus \{x, y\}$.  Let $X'=X\setminus\{x\}$,
$Y'=Y\setminus\{y\}$, and $H'= H\setminus\{x,y\}$.  Clearly every
vertex $x'\in X'$ satisfies $|L'(x')|\ge |X'|$ and every vertex $y'\in
Y'$ satisfies $|L'(y')|\ge |Y'|$, and $|X'|\le |Y'|$, and there are
$|X'|$ non-edges between $X'$ and $Y'$, and they form a matching in
$\overline{H'}$.  By the induction hypothesis, $H'$ admits an
$L'$-coloring.  We can extend it to an $L$-coloring of $H$ by
assigning the color $c$ to $x$ and $y$.
\end{proof}

\begin{lemma}\label{lemma:x-2-y-2}
Let $H$ be a cobipartite graph, where $V(H)$ is partitioned into two
cliques $X = \{x_1, x_2\}$ and $Y = \{y_1, y_2\}$, and
$E(\overline{H})= \{x_2y_2\}$.  Let $L$ be a list assignment on $V(H)$
such that $|L(u)| \ge 2$ for all $u \in V(H)$.  Then $H$ is
$L$-colorable if and only if every clique $Q$ of $H$ satisfies $|L(Q)|
\geq |Q|$.
\end{lemma}
\begin{proof}
This is a corollary of Claim~1
in~\cite{Gravier:1997:CNE:249460.249532}.  For completeness, we
restate the claim here: {\it The graph $H$ is not $L$-colorable if and
only if for some $v\in\{x_2,y_2\}$ we have $L(x_1) = L(y_1) = L(v)$
and these three lists are of size two.}
	
Clearly, if $H$ is $L$-colorable, then every clique $Q$ of $H$
satisfies $|L(Q)| \geq |Q|$.  Conversely, if every clique $Q$ of $H$
satisfies $|L(Q)| \geq |Q|$, then by the above claim, applied to the
cliques $\{x_1,y_1,x_2\}$ and $\{x_1,y_1,y_2\}$, we obtain that $H$ is
$L$-colorable.
\end{proof}

\begin{lemma}\label{lemma:x-3-y-2}
Let $H$ be a cobipartite graph, where $V(H)$ is partitioned into two
cliques $X = \{x_1, x_2, x_3\}$ and $Y = \{y_1, y_2\}$, and
$E(\overline{H})= \{x_3y_2\}$.  Let $L$ be a list assignment on $V(H)$
such that $|L(x)| \ge 3$ for all $x \in X$ and $|L(y)| \ge 2$ for all
$y \in Y$.  Then $H$ is $L$-colorable if and only if every clique $Q$
of $H$ satisfies $|L(Q)|\ge |Q|$.
\end{lemma}
\begin{proof}
If $H$ is $L$-colorable then clearly every clique $Q$ of $H$ satisfies
$|L(Q)|\ge |Q|$. Now let us prove the converse.

First suppose that $L(y_2)\subseteq L(x_3)$.  Since $H\setminus
\{x_3\}$ is a clique, every subset $T$ of $V(H)\setminus\{x_3\}$
satisfies $|L(T)|\ge |T|$, and so, by Hall's theorem there is an
$L$-coloring of $H\setminus \{x_3\}$.  Then we can extend any such
coloring by assigning to $x_3$ the color assigned to $y_2$.

Now assume that $L(y_2)\not\subseteq L(x_3)$.  This implies $|L(x_3)
\cup L(y_2)| \ge 4$.  Suppose that the family $\{L(x) \mid x \in
V(H)\}$ does not have a system of distinct representatives.  By Hall's
theorem there is a set $T\subseteq V(H)$ such that $|L(T)|< |T|$.  By
the assumption, $T$ is not a clique, so it contains $x_3$ and $y_2$.
It follows that $|L(T)|\ge 4$.  Hence $|T|=5$, so $T=V(H)$, and
$|L(T)|=4$, and we may assume that $L(x_3)=\{1,2,3\}$ and
$L(y_2)=\{3,4\}$ and $L(T)=\{1,2,3,4\}$.  Assign color~$3$ to $x_3$
and $y_2$.  Now assign a color $c$ from $L(y_1)\setminus\{3\}$ to
$y_1$ (there may be two choices for $c$).  We may assume that this
coloring fails to be extended to $\{x_1, x_2\}$; so it must be that
$L(x_1)\setminus\{3,c\}$ and $L(x_2)\setminus\{3,c\}$ are equal and of
size~$1$; so $L(x_1)= L(x_2)= \{b,c,3\}$ for some $b\neq c$, with
$b\in\{1,2,4\}$.  Suppose that $3\notin L(y_1)$.  Then there is a
second choice for $c$, and we may assume that this attempt fails
similarly.  Hence $L(y_1)=\{b,c\}$, with $b,c\in\{1,2,4\}$.  If
$\{b,c\}=\{1,2\}$, then the clique $Q_1=\{x_1,x_2,x_3,y_1\}$ violates
the assumption because $L(Q_1)= \{1,2,3\}$.  If $\{b,c\}=\{1,4\}$ or
$\{2,4\}$, then the clique $Q_2=\{x_1,x_2,y_1,y_2\}$ violates the
assumption because $L(Q_2)= \{b,c,3\}$.  So we may assume that $3\in
L(y_1)$, i.e., $L(y_1)= \{c,3\}$.  If $c=4$, then $Q_2$ violates the
assumption because $L(Q_2)=\{b,3,4\}$.  So, up to symmetry, $c=1$.  If
$b=2$, then $Q_1$ violates the assumption because $L(Q_1)=\{1,2,3\}$.
If $b=4$, then $Q_2$ violates the assumption because
$L(Q_2)=\{1,3,4\}$.  Hence the family $\{L(x) \mid x \in V(H)\}$
admits a system of distinct representatives, which is an $L$-coloring
of $G$.
\end{proof}

\begin{lemma}\label{lemma:x-3-y-3}
Let $H$ be a cobipartite graph, where $V(H)$ is partitioned into two
cliques $X = \{x_1, x_2, x_3\}$ and $Y = \{y_1, y_2, y_3\}$, and
$E(\overline{H}) =\{x_2y_2, x_3y_3\}$.  Let $L$ be a list assignment
on $V(H)$ such that $|L(x)| \ge 3$ for all $x \in V(H)$.  Then $H$ is
$L$-colorable if and only if every clique $Q$ of $H$ satisfies
$|L(Q)|\ge |Q|$.  In particular, if $|L(x_1) \cup L(y_1)|\ge 4$, then
$H$ is $L$-colorable.
\end{lemma}
\begin{proof}
If $H$ is $L$-colorable then clearly every clique $Q$ of $H$ satisfies
$|L(Q)|\ge |Q|$.  Now let us prove the converse.  We first claim that:
\begin{equation}\label{lxiyi2}
\mbox{We may assume that $|L(x_i) \cap L(y_i)| \le 1$ for each 
$i\in\{2,3\}$.}
\end{equation}
Suppose on the contrary, and up to symmetry, that $|L(x_2) \cap
L(y_2)|\ge 2$.  Let $H'= H \setminus \{x_2\}$, and set $L'(y_2)=
L(x_2)\cap L(y_2)$ and $L'(u)=L(u)$ for all $u\in
\{x_1,x_3,y_1,y_3\}$.  Thus $H'$ and $L'$ satisfy the hypothesis of
Lemma~\ref{lemma:x-3-y-2}.  If every clique $Q$ in $H'$ satifies
$|L'(Q)|\ge |Q|$, then Lemma~\ref{lemma:x-3-y-2} implies that $H'$
admits an $L'$-coloring, and we can extend it to an $L$-coloring of
$H$ by giving to $x_2$ the color assigned to $y_2$.  Hence assume that
some clique $Q$ in $H'$ satisfies $|L'(Q)|< |Q|$.  We have $|L'(Q)|\ge
2$, so $|Q|\ge 3$, so $3\le |L'(Q)|< |Q|\le 4$, and so $|L'(Q)|=3$ and
$|Q|=4$.  Since $x_3$ and $y_3$ play symmetric roles here, we may
assume up to symmetry that $Q=\{x_1,y_1,y_2,y_3\}$, and
$L'(Q)=\{a,b,c\}$, where $a,b,c$ are three distinct colors.  Hence
$L(x_1)= L(y_1)= L(y_3)=\{a,b,c\}$.  Since $|L(Q)|\ge 4$, there is a
color $d\in L(y_2)\setminus\{a,b,c\}$.  Since
$|L(\{x_1,y_1,x_2,y_3\})|\ge 4$, there is a color $e\in
L(y_2)\setminus\{a,b,c\}$.  If $a\in L(x_3)$, then we can assign color
$a$ to $x_3$ and $y_3$, colors $b$ and $c$ to $x_1$ and $y_1$, color
$e$ to $x_2$ and color $d$ to $y_2$.  So assume that $a\notin L(x_3)$,
and similarly that $b,c\notin L(x_3)$.  Then we can assign colors
$a,b,c$ to $x_1$, $y_1$, $y_3$, color $e$ to $x_2$, color $d$ to
$y_2$, and a color from $L(x_3)\setminus\{d,e\}$ to $x_3$.  Thus
(\ref{lxiyi2}) holds.

\medskip

It follows from (\ref{lxiyi2}) that $|L(x_i)\cup L(y_i)|\ge 5$ for
$i=2,3$.  If the family $\{L(x) \mid x \in V(H)\}$ admits a system of
distinct representatives, then this is an $L$-coloring.  So suppose
the contrary.  By Hall's theorem there is a set $T\subseteq V(H)$ such
that $|L(T)|< |T|$.  By the assumption, $T$ is not a clique, so it
contains $x_i$ and $y_i$ for some $i\in\{2,3\}$.  By (\ref{lxiyi2}) we
have $|L(T)|\ge 5$, so $|T|\ge 6$, hence $T=V(H)$, and $|L(T)|=5$, and
consequently $|L(x_i)|=|L(y_i)|=3$ and $|L(x_i)\cap L(y_i)|= 1$ for
each $i=2,3$.  Let $L(x_i) \cap L(y_i) = \{c_i\}$ for $i = 2, 3$.

\medskip

Suppose that $c_2 \neq c_3$.  We assign color $c_i$ to $x_i$ and $y_i$
for each $i=2,3$.  If this coloring can be extended to $\{x_1, y_1\}$
we are done.  So suppose the contrary.  Then it must be that $L(x_1)=
L(y_1)= \{b,c_2,c_3\}$ for some color $b\in L(H)\setminus\{c_2,c_3\}$.
Then we can color $H$ as follows.  Assign colors $c_2$ and $c_3$ to
$x_1$ and $y_1$.  There are four ways to color $x_2$ and $y_2$ with
one color from $L(x_2)\setminus\{c_2\}$ for $x_2$ and one color from
$L(y_2)\setminus\{c_2\}$ for $y_2$; at most two of them use a pair of
colors equal to $L(x_3)\setminus\{c_3\}$ or $L(y_3)\setminus\{c_3\}$,
so we can choose another way, and there will remain a color for $x_3$
and a color for $y_3$.

\medskip

Now suppose that $c_2 = c_3$; call this color $c$.  Let $L'(v) = L(v)
\setminus \{c\}$ for all $v \in V(H) \setminus \{x_3, y_3\}$.  We may
assume that the graph $H \setminus \{x_3, y_3\}$ does not admit an
$L'$-coloring, for otherwise such a coloring can be extended to $H$ by
assigning color $c$ to $x_3$ and $y_3$.  Hence, by
Lemma~\ref{lemma:x-2-y-2} there is a clique $Q$ of size~$3$ in $H
\setminus \{x_3, y_3\}$ such that $|L'(Q)|=2$, say $L'(Q)=\{a,b\}$.
So $L(u) = \{a, b, c\}$ for all $u \in Q$.  Moreover $Q$ consists of
$x_1,y_1$ and one of $x_2,y_2$.  We assign color $a$ to $x_1$, color
$b$ to $y_1$, and color $c$ to $x_2$ and $y_2$.  Since
$|L(Q\cup\{x_3\})|\ge 4$, there is a color $d\in
L(x_3)\setminus\{a,b,c\}$, and similarly there is a color $e\in
L(y_3)\setminus\{a,b,c\}$.  We assign $d$ to $x_3$ and $e$ to $y_3$,
and we obtain an $L$-coloring of $H$.

\medskip

Finally we prove the last sentence of the lemma.  Since $x_1$ and
$y_1$ are in all cliques of size $4$, the assumption that $|L(x_1)
\cup L(y_1)|\ge 4$ implies that every clique $Q$ of $H$ satisfies
$|L(Q)| \geq |Q|$.  So $H$ is $L$-colorable.
\end{proof}

\begin{lemma}\label{lem:c2}
Let $H$ be a cobipartite graph with $\omega(H)\le 4$.  Let $x,y$ be
two adjacent vertices in $H$ such that $N(x)\setminus\{y\}$ and
$N(y)\setminus\{x\}$ are cliques and $V(H)=N(x)\cup N(y)$.  Let $L$ be
a list assignment such that $|L(x)|\ge 2$, $|L(y)|\ge 2$, and
$|L(v)|\ge 4$ for all $v\in V(H)\setminus \{x,y\}$.  Then $H$ is
$L$-colorable.
\end{lemma}
\begin{proof}
Let $X=N(x)\setminus\{y\}$ and $Y=N(y)\setminus\{x\}$.  Let $I=X\cap
Y$.  Since $\{x,y\}\cup I$ is a clique, we have $|I|\le 2$.

\medskip

First suppose that $|I|=2$.  Let $I=\{w,w'\}$.  Since $\{x\}\cup X$ is
a clique that contains $I$, we have $|X\setminus I|\le 1$.  Likewise
$|Y\setminus I|\le 1$.  We may assume that we are in the situation
where $X\setminus I$ and $Y\setminus I$ are non-empty and complete to
each other, because any other situation can be reduced to that one by
adding vertices or edges (which makes the coloring problem only
harder).  Let $X\setminus I=\{u\}$ and $Y\setminus I=\{v\}$.  Suppose
that $L(x)\cap L(v)\neq\emptyset$.  Pick a color $a\in L(x)\cap L(v)$,
assign it to $x$ and $v$, and remove it from the lists of all other
vertices.  Pick a color $b$ from $L(y)\setminus\{a\}$, assign it to
$y$ and remove it from the list of the vertices in $I$.  Let $L'$ be
the reduced list assignment.  Then $|L'(w)|\ge 2$, $|L'(w')|\ge 2$,
and $|L'(u)|\ge 3$, so we can $L'$-color greedily $w,w',u$ in this
order.  Hence assume that $L(x)\cap L(v)=\emptyset$, and similarly
that $L(y)\cap L(u)=\emptyset$.  Then $|L(x)\cup L(v)|\ge 6$ and
$|L(y)\cup L(u)|\ge 6$.  It follows that the family $\{L(z)\mid z\in
V(H)\}$ satisfies Hall's condition, so $H$ is $L$-colorable.

\medskip

Now suppose that $|I|=1$.  Let $I=\{w\}$.  Then $|X\setminus \{w\}|\le
2$ and $|Y\setminus \{w\}|\le 2$.  We may assume that we are in the
situation where $X\setminus I$ and $Y\setminus I$ have size $2$ and
there are three edges between them, because any other situation can be
reduced to that one by adding vertices or edges.  Let $X\setminus
I=\{u,v\}$ and $Y\setminus I=\{s,t\}$, and let $us,ut,vs\in E(H)$ and
$vt\notin E(H)$.  Suppose that $L(x)\cap L(s)\neq\emptyset$.  We pick
a color $a\in L(x)\cap L(s)$, assign it to $x$ and $s$, and remove it
from the lists of all other vertices.  Then it is easy to see that we
can color $y,t,w,u,v$ in this order, using colors from the reduced
lists.  Hence assume that $L(x)\cap L(s)=\emptyset$, and similarly
that $L(y)\cap L(u)=\emptyset$.  So $|L(x)\cup L(s)|\ge 6$ and
$|L(y)\cup L(u)|\ge 6$.  \\
Suppose that $L(x)\cap L(t)\neq\emptyset$.  We pick a color $a\in
L(x)\cap L(t)$, assign it to $x$ and $t$, and remove it from the lists
of all other vertices.  Since $L(x)\cap L(s)=\emptyset$, the list
$L(s)$ loses no color ($a\notin L(s)$).  If $L(y)\setminus\{a\}$ and
$L(v)\setminus \{a\}$ have a common element $b$, we assign it to $y$
and $v$, and it is easy to see that $w,u,s$ can be colored in this
order with the reduced lists.  On the other hand if
$L(y)\setminus\{a\}$ and $L(v)\setminus \{a\}$ are disjoint, then it
is easy to see that the family $\{L(z)\setminus\{a\}\mid z\in
V(H)\setminus\{x,t\}\}$ satisfies Hall's condition, so $H$ is
$L$-colorable.  Hence assume that $L(x) \cap L(t) = \emptyset$, and
similarly that $L(y) \cap L(v) = \emptyset$.  So $|L(x)\cup L(t)|\ge
6$ and $|L(y)\cup L(v)|\ge 6$.  \\
Suppose that $L(t) \cap L(v) \neq \emptyset$.  Pick a color $a \in
L(t) \cap L(v)$ and assign it to $t$ and $v$.  Since $L(y) \cap L(v) =
\emptyset$ and $L(x) \cap L(t) = \emptyset$ we have $L(y) = L(y)
\setminus \{a\}$ and similarly $L(x) = L(x) \setminus \{a\}$.  It
follows that the family $\{L(z) \setminus \{a\} \mid z \in V(H)
\setminus \{t, v\}\}$ satisfies Hall's condition.  Finally assume that
$L(t) \cap L(v) = \emptyset$.  So $|L(t)\cup L(v)|\ge 8$.  Then the
family $\{L(z) \mid z \in V(H)\}$ satisfies Hall's condition, so $H$
is $L$-colorable.

\medskip

Finally suppose that $I=\emptyset$.  We may assume that $X$ and $Y$
have size $3$ and that the non-edges between them form a matching of
size $2$, because any other situation can be reduced to that one by
adding vertices or edges.  Let $X = \{u_1, u_2, u_3\}$, $Y = \{v_1,
v_2, v_3\}$, and $E(\overline{H})= \{u_2v_2, u_3v_3\}$.  We can choose
a color $a$ from $L(x)$ and a color $b$ from $L(y)$ such that $L(u_1)
\setminus \{a\} \neq L(v_1) \setminus \{b\}$.  Let $L'(u) = L(u)
\setminus \{a\}$ for all $u \in X$ and $L'(v) = L(v) \setminus \{b\}$
for all $v \in Y$.  By the last sentence of Lemma~\ref{lemma:x-3-y-3},
$H \setminus \{x,y\}$ admits an $L'$-coloring, and we can extend it to
an $L$-coloring of $H$ by assigning color $a$ to $x$ and color $b$ to
$y$.
\end{proof}

\begin{lemma}\label{lem:c3elem}
Let $H$ be a cobipartite graph, where $V(H)$ is partitioned into two
cliques $X = \{x_1, x_2, x_3\}$ and $Y = \{y_1, y_2, y_3\}$, and
$E(\overline{H})= \{x_1y_1, x_2y_2, x_3y_3,$ $x_3y_1,$ $x_1y_2\}$.
Let $L$ be a list assignment on $V(H)$ such that $|L(x_3)| = 2$,
$|L(y_2)| = 2$, and $|L(w)| = 3$ for every $w \in V(H) \setminus
\{x_3, y_2\}$.  Then $H$ is $L$-colorable.
\end{lemma}

\begin{proof}
Suppose that $L(x_2) \cap L(y_2) \neq \emptyset$.  Assign a color $a$
from $L(x_2) \cap L(y_2)$ to $x_2$ and $y_2$.  Let $L'(u) = L(u)
\setminus \{a\}$ for all $u\in \{x_1,x_3,y_1,y_3\}$.  Then we can
$L'$-color $x_3, x_1, y_3, y_1$ greedily in this order, because
$x_3$-$x_1$-$y_3$-$y_1$ is an induced path and the reduced lists' size
pattern is $(\ge 1, \ge 2, \ge 2, \ge 2)$.  The proof is similar when
$L(x_3) \cap L(y_3) \neq \emptyset$.  So we may assume that:
\begin{equation}\label{c3elem-eq1}
\mbox{$L(x_2) \cap L(y_2) = \emptyset$ and $L(x_3) \cup L(y_3) =
\emptyset$.}
\end{equation}

Suppose that $L(x_1) \cap L(y_2) \neq \emptyset$.  Assign a color $a$
from $L(x_1) \cap L(y_2)$ to $x_1$ and $y_2$.  Let
$L'(u)=L(u)\setminus\{a\}$ for all $u\in\{x_2,x_3,y_1,y_3\}$.  By
(\ref{c3elem-eq1}), we have $a\notin L(x_2)$, so $L'(x_2)=L(x_2)$, and
$a$ is in at most one of $L(x_3)$ and $L(y_3)$.  If $a\in L(x_3)$,
then we can $L'$-color greedily $x_3$, $x_2$, $y_1$, $y_3$ in this
order.  If $a\in L(y_3)$, then we can $L'$-color greedily $y_3$,
$y_1$, $x_2$, $x_3$ in this order.  The proof is similar when $L(x_3)
\cap L(y_1) \neq \emptyset$.  So we may assume that:
\begin{equation}\label{c3elem-eq2}
\mbox{$L(x_1) \cap L(y_2) = \emptyset$ and $L(x_3) \cap L(y_1) =
\emptyset$.}
\end{equation}

Suppose that $L(x_1) \cap L(y_1) \neq \emptyset$.  Assign a color $a$
from $L(x_1) \cap L(y_1)$ to $x_1$ and $y_1$.  Let
$L'(u)=L(u)\setminus\{a\}$ for all $u\in\{x_2,x_3,y_2,y_3\}$.  By
(\ref{c3elem-eq2}), we have $a\notin L(x_3)$ and $a\notin L(y_2)$.
The graph $H\setminus\{x_1,y_1\}$ is an even cycle, and $|L'(u)|\ge 2$
for every vertex $u$ in that graph, so it is $L'$-colorable.  So we
may assume that:
\begin{equation}\label{c3elem-eq3}
\mbox{$L(x_1) \cap L(y_1) = \emptyset$.}
\end{equation}
   
By (\ref{c3elem-eq1}), (\ref{c3elem-eq2}) and (\ref{c3elem-eq3}), we
have $|L(u) \cup L(v)| = 5$ whenever $\{u,v\}$ is any of
$\{x_2,y_2\}$, $\{x_3,y_3\}$, $\{x_1,y_2\}$, $\{x_3,y_1\}$, and
$|L(x_1) \cap L(y_1)| = 6$.  It follows that the family $\{L(w) \mid w
\in V(H)\}$ admits a system of distinct representatives, which is an
$L$-coloring for $H$.
\end{proof}

\begin{lemma}\label{lem:c3b}
Let $H$ be a cobipartite graph with $\omega(G)\le 4$.  Let $V(H)$ be
partitioned into two cliques $X,Y$ with $X=\{x_1,x_2,x_3\}$, such that
$x_1$ is complete to $Y$.  Let $L$ be a list assignment such that
$|L(x_1)|\ge 3$, $|L(x_2)|\ge 2$, $|L(x_3)|\ge 2$, and $|L(y)|\ge 4$
for all $y\in Y$.  Then $H$ is $L$-colorable.
\end{lemma}
\begin{proof}
Since $Y\cup\{x_1\}$ is a clique, we have $|Y|\le 3$.  If $|Y| \leq
2$, then Lemma~\ref{lemma:x-3-y-2} implies that $H$ is $L$-colorable.
So we may assume that $|Y| = 3$, say $Y = \{y_1, y_2, y_3\}$, and we
may assume that $E(\overline{H})= \{x_2y_2,x_3y_3\}$.  If the family
$\{L(w) \mid w \in V(H)\}$ admits a system of distinct
representatives, then this is an $L$-coloring of $H$, so assume the
contrary.  So there is a set $T\subseteq V(H)$ such that $|L(T)|<|T|$.
We have $|L(T)|\ge 2$, so $|T|\ge 3$, so $|L(T)|\ge 3$, so $|T|\ge 4$,
so $T\cap Y\neq\emptyset$, so $|L(T)|\ge 4$, and so $|T|\ge 5$.  It
follows that $T$ is not a clique.  Hence assume that $x_2,y_2\in T$.
If $L(x_2)\cap L(y_2)=\emptyset$, then $|L(T)|\ge |L(x_2) \cup L(y_2)|
= 6$, so $|T|\ge 7$, which is impossible.  Hence $L(x_2) \cap L(y_2)
\neq \emptyset$.  Assign a color $c_2$ from $L(x_2) \cap L(y_2)$ to
$x_2$ and $y_2$.  Define $L'(u) = L(u) \setminus \{c_2\}$ for all
$u\in V(H)\setminus\{x_2,y_2\}$.  If $L'(x_3) \cap L'(y_3) \neq
\emptyset$ assign a color $c_3$ from $L'(x_3) \cap L'(y_3)$ to $x_3$
and $y_3$.  Then we have $|(L'(x_1) \cup L'(y_1)) \setminus \{c_2\}|
\geq 2$, so we can extend the coloring to $\{x_1,y_1\}$.  On the other
hand, if $L'(x_3) \cap L'(y_3) = \emptyset$, the family $\{L'(w) \mid
w \in V(H) \setminus \{x_2, y_2\}\}$ admits a system of distinct
representatives.  So $H$ admist an $L$-coloring.
\end{proof}

\begin{lemma}\label{lem:c4}
Let $H$ be a cobipartite graph, where $V(H)$ is partitioned into two
cliques $X = \{x_1, x_2, x_3, x_4\}$ and $Y = \{y_1, y_2, y_3, y_4\}$,
and $E(\overline{H})= \{x_1y_1, x_1y_3,$ $x_1y_4,$ $x_2y_2, x_2y_3,
x_2y_4, x_3y_3, x_4y_4\}$.  Let $L$ be a list assignment on $V(H)$
such that $|L(x_1)| = 2$, $|L(x_2)| = 2$ and $|L(w)| = 4$ for all $w
\in V(H) \setminus \{x_1, x_2\}$.  Then $H$ is $L$-colorable.
\end{lemma}

\begin{proof}
We choose colors $c_1,c_2$ with $c_1\in L(x_1)$, $c_2\in L(x_2)$ and
$c_1\neq c_2$, such that if $|L(y_1)\cap L(y_2)|=3$, then either
$\{c_1\}\neq L(y_2)\setminus L(y_1)$ or $\{c_2\}\neq L(y_1)\setminus
L(y_2)$.  This is possible as follows: if $|L(y_1)\cap L(y_2)|=3$, let
$\alpha$ be the color in $L(y_1)\setminus L(y_2)$, then choose $c_2\in
L(x_2)\setminus\{\alpha\}$ and $c_1\in L(x_1)\setminus\{c_2\}$.  We
assign color $c_1$ to $x_1$ and $c_2$ to $x_2$.  Let $L'(y_1) = L(y_1)
\setminus \{c_2\}$, $L'(y_2) = L(y_2) \setminus \{c_1\}$, $L'(x_3) =
L(x_3) \setminus \{c_1, c_2\}$, $L'(x_4) = L(x_4) \setminus \{c_1,
c_2\}$, $L'(y_3) = L(y_3)$ and $L'(y_4) = L(y_4)$.  So $|L'(u)|\ge 2$
for $u\in \{x_3,x_4\}$, $|L'(v)|\ge 3$ for $v\in \{y_1,y_2\}$, and
$|L'(w)|=4$ for $w\in \{y_3,y_4\}$.  Note that the choice of $c_1$ and
$c_2$ implies that $|L'(y_1)\cup L'(y_2)|\ge 4$.  Now we show that
$H\setminus\{x_1,x_2\}$ is $L'$-colorable.

\medskip

Suppose that $L'(x_3) \cap L'(y_3) \neq \emptyset$.  Assign a color
$c_3$ from $L'(x_3) \cap L'(y_3)$ to $x_3$ and $y_3$.  Define $L''(u)
= L'(u) \setminus \{c_3\}$ for all $u\in\{x_4, y_1, y_2, y_4\}$.  Note
that $|L''(x_4)|\ge 1$, $|L''(u)|\ge 2$ for $u\in\{y_1,y_2\}$, and
$|L''(y_4)|\ge 3$.  Assign a color $c_4$ from $L''(x_4)$ to $x_4$.
Since $|L'(y_1)\cup L'(y_2)|\ge 4$, it follows that $|(L''(y_1)\cup
L''(y_2))\setminus\{c_4\}|\ge 2$.  So we can $L''$-color greedily
$\{y_1,y_2\}$ and then $y_4$.  The proof is similar if $L'(x_4) \cap
L'(y_4) \neq \emptyset$.  Therefore we may assume that $L'(x_3) \cap
L'(y_3) = \emptyset$ and $L'(x_4) \cap L'(y_4) = \emptyset$, and so
$|L'(x_3) \cup L'(y_3)| = 6$ and $|L'(x_4) \cup L'(y_4)| =6$.  This
and the choice of $c_1$, $c_2$ implies that the family $\{L'(w) \mid w
\in V(H) \setminus \{x_1, x_2\}\}$ admits a system of distinct
representatives.
\end{proof}

\begin{lemma}\label{lem:c3}
Let $H$ be a cobipartite graph with $\omega(G)\le 4$.  Let $C$ be a
clique of size $3$ in $H$ such that for every $w\in C$, the set
$N(w)\setminus C$ is a clique.  Let $L$ be a list assignment such that
$|L(w)|= 3$ for all $w\in C$ and $|L(v)|= 4$ for all $v\in
V(H)\setminus C$.  Then $H$ is $L$-colorable.
\end{lemma}
\begin{proof}
If $H$ is not connected, it has two components $H_1,H_2$ and both are
cliques of size at most $4$.  The hypothesis implies easily that for
each $i\in\{1,2\}$ the family $\{L(u)\mid u\in V(H_i)\}$ satisfies
Hall's theorem, and consequently $H$ is $L$-colorable.  Hence we
assume that $H$ is connected.  Let $n=|V(H)|$ and $V(H)=\{v_1, \ldots,
v_n\}$.  The hypothesis implies that $n\le 8$.  Let $\mu=n-4$.  Since
$\omega(H)=4$, K\H{o}nig's theorem implies that $\overline{H}$ has a
matching of size $\mu$.  We may assume that the pairs
$\{v_i,v_{i+\mu}\}$ ($i=1,\ldots,\mu$) form such a matching.  We may
also assume that $E(H)$ is maximal under the hypothesis of the lemma,
since adding edges can only make the problem harder.

\medskip

First suppose that $n=4$.  The hypothesis implies that the family
$\{L(u)\mid u\in V(H)\}$ satisfies Hall's theorem, and consequently
$H$ is $L$-colorable.

\medskip

Now suppose that $n=5$.  So $\mu=1$ and $v_1v_2 \in E(\overline{H})$.
Up to symmetry, we have either $C=\{v_3,v_4,v_5\}$ or
$C=\{v_1,v_3,v_4\}$.  If $C=\{v_3,v_4,v_5\}$, then we can $L$-color
greedily the vertices $v_3, v_4, v_5, v_1, v_2$ in this order.  If
$C=\{v_1,v_3,v_4\}$, then we can $L$-color greedily the vertices $v_1,
v_3, v_4, v_5, v_2$ in this order.

\medskip

Now suppose that $n=6$.  So $\mu=2$ and $\{v_1v_3, v_2v_4\}\subseteq
E(\overline{H})$.  Up to symmetry, we have either $C=\{v_1,v_5,v_6\}$
or $C=\{v_1,v_2,v_5\}$.  Suppose that $C=\{v_1,v_5,v_6\}$.  Since
$\{v_1,v_2,v_4\}$ is not a stable set of size $3$ and $N(v_1)\setminus
C$ is a clique, $v_1$ is adjacent to exactly one of $v_2,v_4$, say to
$v_4$ and not to $v_2$.  Then we can $L$-color greedily the vertices
$v_1, v_5, v_6, v_4, v_3, v_2$ in this order.  Suppose that
$C=\{v_1,v_2,v_5\}$.  By the maximality of $E(H)$ we may assume that
$E(\overline{H})= \{v_1v_2,$ $v_3v_4\}$.  Then
Lemma~\ref{lemma:x-3-y-3} (with $X=C$, $Y=V(H)\setminus C$, $x_1=v_5$
and $y_1=v_6$) implies that $H$ is $L$-colorable.

\medskip

Now suppose that $n=7$.  So $\mu=3$, and $\{v_1v_4, v_2v_5,
v_3v_6\}\subseteq E(\overline{H})$.  Up to symmetry, we have either
$C=\{v_1,v_2,v_3\}$ or $C=\{v_1,v_2,v_7\}$.  If $C=\{v_1,v_2,v_3\}$,
then, by the maximality of $E(H)$ we may assume that $E(\overline{H})=
\{v_1v_4$, $v_2v_5$, $v_3v_6\}$, and by Lemma~\ref{lemma:x-y-4} (with
$X=C$ and $Y=V(H)\setminus C$), $H$ is $L$-colorable.  So suppose that
$C=\{v_1,v_2,v_7\}$.  For each $i\in\{1,2\}$, $v_i$ has exactly one
neighbor in $\{v_3,v_6\}$, for otherwise either $\{v_i,v_3,v_6\}$ is a
stable set of size $3$ or $N(v_i)\setminus C$ is not a clique.  This
leads to the following two cases (a) and (b): 

(a) $v_1$ and $v_2$ have the same neighbor in $\{v_3,v_6\}$.  We may
assume that $v_1v_3,$ $v_2v_3\in E(H)$ and $v_1v_6,v_2v_6\notin E(H)$.
Since $H$ is cobipartite, $\{v_1,v_2,v_3\}$ and $\{v_4,v_5,v_6\}$ are
cliques, and by the maximality of $E(H)$ we may assume that
$\{v_1v_5,$ $v_2v_4,$ $v_3v_4,$ $v_3v_5\}\subseteq E(H)$ and that
$v_7$ is complete to $\{v_1,\ldots,v_6\}$.  Pick a color $c$ from
$L(v_7)$, assign it to $v_7$, and set $L'(u)=L(u)\setminus \{c\}$ for
all $u\in V(H)\setminus\{v_7\}$.  By Lemma~\ref{lemma:x-y-4} (with
$X=\{v_1,v_2\}$ and $Y=\{v_3,v_4,v_5\}$), $H\setminus\{v_6,v_7\}$
admits an $L'$-coloring.  This can be extended to $v_6$ since $v_6$
has only two neighbors in $H\setminus\{v_7\}$.  So $H$ is
$L$-colorable.  

(b) $v_1$ and $v_2$ do not have the same neighbor in $\{v_3,v_6\}$.
We may assume that $v_1v_3, v_2v_6\in E(H)$ and $v_1v_6,v_2v_3\notin
E(H)$.  Since $H$ is cobipartite, $\{v_1,v_3,v_5\}$ and
$\{v_2,v_4,v_6\}$ are cliques, and by the maximality of $E(H)$ we may
assume that $v_4v_5,v_5v_6\in E(H)$ and that $v_7$ is complete to
$\{v_1,\ldots,v_6\}$.  Pick a color $c$ from $L(v_7)$, assign it to
$v_7$, and set $L'(u)=L(u)\setminus \{c\}$ for all $u\in
V(H)\setminus\{v_7\}$.  By Lemma~\ref{lem:c3elem}, $H\setminus\{v_7\}$
is $L'$-colorable.  So $H$ is $L$-colorable.

\medskip

Now suppose that $n=8$.  So $\mu=4$ and $\{v_1v_5, v_2v_6, v_3v_7,
v_4v_8\}\subseteq E(\overline{H})$.  Up to symmetry we have
$C=\{v_1,v_2,v_3\}$.  For each $i\in\{1,2,3\}$, $v_i$ has exactly one
neighbor in $\{v_4,v_8\}$, for otherwise either $\{v_i,v_4,v_8\}$ is a
stable set of size $3$ or $N(v_i)\setminus C$ is not a clique.  This
leads to two cases: (a) $v_1, v_2,v_3$ have the same neighbor in
$\{v_4,v_8\}$; (b) only two of $v_1,v_2,v_3$ have a common neighbor in
$\{v_4,v_8\}$.  

Suppose that (a) holds.  We may assume that $v_1,v_2,v_3$ are all
adjacent to $v_4$ and not adjacent to $v_8$.  Since $H$ is
cobipartite, $\{v_1,\ldots,v_4\}$ and $\{v_5,\ldots,v_8\}$ are
cliques, and by the maximality of $E(H)$ we may assume that
$E(\overline{H})=\{v_1v_5,v_2v_6,v_3v_7,v_4v_8,v_1v_8,v_2v_8,v_3v_8\}$.
By Lemma~\ref{lemma:x-y-4} (with $X=\{v_1,v_2,v_3\}$ and
$Y=\{v_4,v_5,v_6,v_7\}$), $H\setminus\{v_8\}$ admits an $L'$-coloring.
This can be extended to $v_8$ since $v_8$ has only three neighbors in
$H$.  So $H$ is $L$-colorable.

Therefore we may assume that (b) holds.  We may assume that $v_1v_4,$
$v_2v_4,$ $v_3v_8\in E(H)$ and $v_1v_8,v_2v_8, v_3v_4\notin E(H)$.
Since $H$ is cobipartite, $\{v_1,v_2,v_4,v_7\}$ and $\{v_3,v_5,
v_6,v_8\}$ are cliques, and by the maximality of $E(H)$ we may assume
that $E(\overline{H})= \{v_1v_5, v_2v_6, v_3v_7, v_4v_8, v_1v_8,v
_2v_8, v_3v_4\}$.

\medskip

Suppose that $L(v_3) \cap L(v_7) \neq \emptyset$.  Assign a color $c$
from $L(v_3) \cap L(v_7)$ to $v_3$ and $v_7$.  Define $L'(w) = L(w)
\setminus \{c\}$ for every $w \in V(H) \setminus \{v_3, v_7\}$.  By
Lemma~\ref{lemma:x-3-y-2}, $H \setminus \{v_3,v_7,v_8\}$ admits an
$L'$-coloring.  This can be extended to $v_8$ since $v_8$ has only two
neighbors in $H\setminus\{v_3,v_7\}$.  So we may assume that:  
\begin{equation}\label{c3-eq2}
\mbox{$L(v_3) \cap L(v_7) =\emptyset$.}
\end{equation}

Suppose that $L(v_1) \cap L(v_5) \neq \emptyset$.  Assign a color $c$
from $L(v_1) \cap L(v_5)$ to $v_1$ and $v_5$.  Define $L'(w) = L(w)
\setminus \{c\}$ for every $w \in V(H) \setminus \{v_1, v_5\}$.  By
Lemma~\ref{lem:c3elem} the graph $H \setminus \{v_1, v_5\}$ is
$L'$-colorable.  The proof is similar if $L(v_2)\cap L(v_6)\neq
\emptyset$.  So we may assume that:
\begin{equation}\label{c3-eq3}
\mbox{$L(v_1) \cap L(v_5) = \emptyset$ and $L(v_2) \cup L(v_6) =
\emptyset$.}
\end{equation}

Suppose that $L(v_3) \cap L(v_4) \neq \emptyset$.  Assign a color $c$
from $L(v_3) \cap L(v_4)$ to $v_3$ and $v_4$.  Define $L'(w) = L(w)
\setminus \{c\}$ for every $w \in V(H) \setminus \{v_3, v_4\}$.  By
(\ref{c3-eq2}), we have $c\notin L(v_7)$, so $L'(v_7) = L(v_7)$.
Hence and by (\ref{c3-eq2}) and (\ref{c3-eq3}), the family $\{L'(w)
\mid w \in V(H) \setminus \{v_3, v_4\}\}$ admits a system of distinct
representatives.  So we may assume that:
\begin{equation}\label{c3-eq4}
\mbox{$L(v_3) \cup L(v_4) = \emptyset$.}
\end{equation}

Suppose that $L(v_4) \cap L(v_8) \neq \emptyset$.  Assign a color $c$
from $L(v_4) \cap L(v_8)$ to $v_4$ and $v_8$.  Define $L'(w) = L(w)
\setminus \{c\}$ for every $w \in V(H) \setminus \{v_4, v_8\}$.  By
(\ref{c3-eq4}), we have $c\notin L(v_3)$, so $L'(v_3) = L(v_3)$.  By
(\ref{c3-eq2}), (\ref{c3-eq3}) and (\ref{c3-eq4}), the family $\{L'(w)
\mid w \in V(H) \setminus \{v_4, v_8\}\}$ admits a system of distinct
representatives.  So we may assume that:
\begin{equation}\label{c3-eq5}
\mbox{$L(v_4) \cup L(v_8) = \emptyset$.}
\end{equation}

By (\ref{c3-eq2}), (\ref{c3-eq3}), (\ref{c3-eq4}) and (\ref{c3-eq5}),
we have $|L(v_i) \cup L(v_j)| = 7$ if the pair $\{i,j\}$ is any of
$\{1,5\}$, $\{2,6\}$, $\{3,7\}$ and $\{3,4\}$, and $|L(v_4) \cup
L(v_8)| = 8$.  It follows easily that the family $\{L(w) \mid w \in
V(H)\}$ admits a system of distinct representatives.
\end{proof}

\section{Elementary graphs}

Now we can consider the case of any elementary graph $G$ with
$\omega(G) \leq 4$.

\begin{theorem}\label{thm-elementary-omega4}
Let $G$ be an elementary graph with $\omega(G) \leq 4$. Then $ch(G) =
\chi(G)$.
\end{theorem}
\begin{proof}
This theorem holds for every graph $G$ with $\omega(G)\le 3$ as proved
in \cite{Gravier2004211}.  Hence we will assume that $\omega(G)=4$.
By Theorem~\ref{thm:mafree}, $G$ is the augmentation of the line-graph
${\cal L}(H)$ of a bipartite multigraph $H$.  Let $e_1, \ldots, e_h$
be the flat edges of ${\cal L}(H)$ that are augmented to obtain $G$.
We prove the theorem by induction on $h$.  If $h = 0$, then $G={\cal
L}(H)$; in that case the equality $ch(G) = \chi(G)$ follows from
Galvin's theorem \cite{Galvin1995153}.  Now assume that $h>0$ and that
the theorem holds for elementary graphs obtained by at most $h-1$
augmentations.  Let $(X,Y)$ be the augment in $G$ that corresponds to
the edge $e_h$ of ${\cal L}(H)$.  In ${\cal L}(H)$, let $e_h=xy$.  So
$x,y$ are incident edges of $H$.  In $H$, let $x=q_xq_{xy}$ and
$y=q_yq_{xy}$; so their common vertex $q_{xy}$ has degree $2$ in $H$.
Let $G_{h-1}$ be the graph obtained from ${\cal L}(H)$ by augmenting
only the $h-1$ other edges $e_1, \ldots, e_{h-1}$.  So $G_{h-1}$ is an
elementary graph.

Let $L$ be a list assignment on $V(G)$ such that $|L(v)|=\omega(G)$
for all $v\in V(G)$.  We will prove that $G$ admits an $L$-coloring.

\begin{equation}\label{eq:x-y-omega}
\mbox{We may assume that $|X \cup Y| > \omega(G)$.}
\end{equation}
Suppose that $|X \cup Y| \le \omega(G)$.  Let $H'$ be the graph
obtained from $H$ by duplicating $|X|-1$ times the edge $x$ (so that
there are exactly $|X|$ parallel edges between the two ends of $x$ in
$H$) and duplicating $|Y|-1$ times the edge $y$.  Let $G_{h-1}'$ be
the graph obtained from ${\cal L}(H')$ by augmenting the $h-1$ edges
$e_1, \ldots, e_{h-1}$ as in $G$.  Then $G_{h-1}'$ can also be
obtained from $G$ by adding all edges between non-adjacent vertices of
$X\cup Y$.  By the assumption, we have $\omega(G'_{h-1})=\omega(G)$.
By the induction hypothesis, $G'_{h-1}$ admits an $L$-coloring.  Then
this is an $L$-coloring of $G$.  Hence (\ref{eq:x-y-omega}) holds.

\medskip

Let $X=\{x_1, \ldots, x_{|X|}\}$ and $Y=\{y_1, \ldots, y_{|Y|}\}$.
Let $N_X=\{v\in V(G)\setminus (X\cup Y)\mid v$ has a neighbor in $X\}$
and $N_Y=\{v\in V(G)\setminus (X\cup Y)\mid v$ has a neighbor in
$Y\}$.  By the definition of a line-graph and of an augment, the set
$N_X$ is a clique and is complete to $X$; hence $|N_X|\le
\omega(G)-|X|$.  Likewise $N_Y$ is a clique and is complete to $Y$,
and $|N_Y|\le \omega(G)-|Y|$.  Let $\mu$ be the size of a maximum
matching in the bipartite graph $\overline{G}[X\cup Y]$.  By
K\H{o}nig's theorem we have $\mu+\omega(G)=|X|+|Y|$, so
$\mu=|X|+|Y|-4$.  Moreover, we may assume that the edges of
$\overline{G}[X\cup Y]$ form a matching of size $\mu$ (for otherwise
we can add some edges to $G$, in $X\cup Y$, which makes the coloring
problem only harder).

The graph $G_{h-1}\setminus\{x,y\}$ is elementary, and it has $h-1$
augments, so, by the induction hypothesis, it admits an $L$-coloring
$f$.  We will try to extend $f$ to $G$; if this fails, we will analyse
why and then show that we can find another $L$-coloring of
$G_{h-1}\setminus\{x,y\}$ that does extend to $G$.  Let $L'$ be the
list assignment defined on $X\cup Y$ as follows: for all $u\in X$, let
$L'(u)=L(u)\setminus f(N_X)$, and for all $v\in Y$, let
$L'(v)=L(v)\setminus f(N_Y)$.  Clearly, $f$ extends to an $L$-coloring
of $G$ if and only if $G[X\cup Y]$ admits an $L'$-coloring.  By
(\ref{eq:x-y-omega}) and up to symmetry, we may assume that either
$|Y|=4$ (and $|X|\le 4$) or $(|X|, |Y|)$ is equal to $(3,3)$ or
$(2,3)$.  We deal with each case separately.

\medskip

{\bf Case 1:} $|Y| = 4$ and $|X|\le 4$.  We have $|N_X|\le 4-|X|$ and
$|N_Y|=0$, so $|L'(u)|\ge |X|$ for all $u\in X$ and $|L'(v)|=4$ for
all $v\in Y$.  Since $\omega(G)=4$, there are $|X|$ non-edges between
$X$ and $Y$ that form a matching in $\overline{G}$.  By
Lemma~\ref{lemma:x-y-4}, $G[X\cup Y]$ admits an $L'$-coloring.

\medskip

{\bf Case 2:} $|X| = |Y| = 3$.  Here we have $\mu=2$, and we may
assume that the non-edges between $X$ and $Y$ are $x_2y_2$ and
$x_3y_3$.  We have $|N_X|\le 1$ and $|N_Y|\le 1$, so $|L'(u)|\ge 3$
for all $u\in X\cup Y$.  If $G[X\cup Y]$ is $L'$-colorable we are
done, so assume the contrary.  By Lemma~\ref{lemma:x-3-y-3}, there is
a clique $Q\subset X\cup Y$ such that $|L'(Q)|<|Q|$.  Thus $3\le
|L'(Q)|<|Q| \le 4$.  This implies that $|Q|=4$, and in particular $Q$
contains $x_1$ and $y_1$.  Moreover $|L'(Q)|=3$, so $L'(x_1)$ and
$L'(y_1)$ are equal and have size $3$, so $|N_X|= 1$ and $|N_Y|= 1$.
Let $N_X=\{u\}$ and $N_Y=\{v\}$.  Thus there are colors $a,b,c,d,d'$
such that $L(x_1) = \{a, b, c, d\}$, $L(y_1) = \{a, b, c, d'\}$,
$f(u)=d$ and $f(v)=d'$ (possibly $d=d'$).  In other words, $f$
satisfies the following ``bad'' property:
\begin{equation}\label{badf}
\longbox{Either $L(x_1)=L(y_1)$ and $f(u)=f(v)$, or $|L(x_1)\cap
L(y_1)|=3$ and $\{f(u)\}=L(x_1)\setminus L(y_1)$ and
$\{f(v)\}=L(y_1)\setminus L(x_1)$.}
\end{equation}

Let $G^*$ be the graph obtained from $G$ by removing all edges between
$X$ and $Y$ and adding two new vertices $u^*$ and $v^*$ with edges
$u^*v^*$, $u^*x_i$ ($i=1,2,3$) and $v^*y_i$ ($i=1,2,3$).  Let $H^*$ be
the graph obtained from $H$ by removing the vertex $q_{xy}$ and adding
three vertices $q_1,q_2,q_3$, with edges $q_1q_2$ and $q_2q_3$, plus
three parallel edges between $q_x$ and $q_1$ and three parallel edges
between $q_3$ and $q_y$.  So $H^*$ is bipartite, and it is easy to see
that $G^*$ is obtained from ${\cal L}(H^*)$ by augmenting $e_1,
\ldots, e_{h-1}$ as in $G$.  So $G^*$ is elementary.

We define a list assignment $L^*$ on $G^*$ as follows.  For all $v\in
V(G\setminus(X\cup Y))$, let $L^*(v)=L(v)$.  For all $v\in X\cup\{u^*,
v^*\}$ let $L^*(v)= \{a, b, c, d\}$, and for all $v\in Y$ let $L^*(v)=
\{a, b, c, d'\}$.  By the induction hypothesis on $h$, the graph $G^*$
admits an $L^*$-coloring $f^*$.  In particular $f^*$ is an
$L$-coloring of $G\setminus(X\cup Y)$.  We claim that if $d=d'$ then
$f^*(u)\neq f^*(v)$, and if $d\neq d'$ then either $f^*(u)\neq d$ or
$f^*(v)\neq d'$.  Indeed we have $f^*(X)=
\{a,b,c,d\}\setminus\{f^*(u)\}$ and $f^*(Y)=
\{a,b,c,d'\}\setminus\{f^*(v)\}$, so if the claim fails then
$f^*(X)=f^*(Y)$ and consequently $f^*(u^*)=f^*(v^*)$, a contradiction.
So the claim holds.  By the claim, we can use $f^*$ instead of $f$
above (as an $L$-coloring of $G\setminus (X\cup Y)$), because $f^*$
does not satisfy (\ref{badf}); so we can extend it to an $L$-coloring
of $G$.

\medskip

{\bf Case 3:} $|X| = 3$ and $|Y| = 2$.  Here we have $\mu=1$, and we
may assume that the only non-edge between $X$ and $Y$ is $x_3y_2$.  We
have $|N_X|\le 1$ and $|N_Y|\le 2$, so $|L'(u)|\ge 3$ for all $u\in X$
and $|L'(v)|\ge 2$ for all $v\in Y$.  If $G[X\cup Y]$ is
$L'$-colorable we are done, so assume the contrary.  By
Lemma~\ref{lemma:x-3-y-2}, there is a clique $Q\subset X\cup Y$ such
that $|L'(Q)|<|Q|$.  This inequality implies that $Q\not\subseteq Y$,
so $Q\cap X\neq\emptyset$.  Thus $3\le |L'(Q)|<|Q| \le 4$.  This
implies that $|Q|=4$, and in particular $Q$ contains $x_1$, $x_2$ and
$y_1$.  Moreover $|L'(Q)|=3$, so $L'(x_1)$ and $L'(x_2)$ are equal and
have size $3$, so $|N_X|= 1$, and $L'(y_1)$ has size at most $3$, so
$|N_Y|\ge 1$, and $L'(y_1) \subseteq L'(x_1)$.  Let $N_X=\{u\}$.  Thus
$L(x_1) = L(x_2)$, and $f$ satisfies the following ``bad'' property:
\begin{equation}\label{badf2}
\mbox{$f(u)\in L(x_1)$ and $L(y_1)\setminus f(N_Y) \subseteq
L(x_1)\setminus \{f(u)\}$.}
\end{equation}

Let $G^*=G\setminus \{x_3\}$.  Clearly $G^*$ is elementary.  Let $H^*$
be the graph obtained from $H$ by duplicating the edge $q_xq_{xy}$ (so
that there are two parallel edges between $q_x$ and $q_{xy}$) and
similarly duplicating $q_yq_{xy}$.  It is easy to see that $G^*$ is
obtained from ${\cal L}(H^*)$ by augmenting $e_1, \ldots, e_{h-1}$ as
in $G$.  We define a list assignment $L^*$ on $G^*$ as follows.  For
all $v\in V(G^*)\setminus \{y_2\}$, let $L^*(v)=L(v)$, and let
$L^*(y_2)= L(y_1)$.  By the induction hypothesis on $h$ the graph
$G^*$ admits an $L^*$-coloring $f^*$.  We claim that $f^*$ does not
satisfy the bad property (\ref{badf2}).  Indeed if it does, then
$f^*(u)\in L^*(x_1)$ and $L^*(y_1)\setminus f^*(N_Y) \subseteq
L^*(x_1)\setminus \{f^*(u)\}$.  Since $L^*(y_2)=L^*(y_1)$, we also
have $L^*(y_2)\setminus f^*(N_Y) \subseteq L^*(x_1)\setminus
\{f^*(u)\}$, and this means that the four vertices $x_1,x_2,y_1,y_2$
(which induce a clique) are colored by $f^*$ using colors from
$L^*(x_1)\setminus \{f^*(u)\}$, which has size $3$; but this is
impossible.  So the claim holds.  By the claim, we can use $f^*$
instead of $f$ above (as an $L$-coloring of $G\setminus (X\cup Y)$)
and we can extend it to an $L$-coloring of $G$.
This completes the proof of the theorem.
\end{proof}

\section{Claw-free perfect graphs}

Now we can prove Theorem~\ref{theorem:claw-free-perfect}, which we
restate here.
\begin{theorem}
Let $G$ be a claw-free perfect graph with $\omega(G) \leq 4$. Then
$ch(G) = \chi(G)$.
\end{theorem}
\begin{proof}
We may assume that $G$ is connected. Let $L$ be a list assignment on
$G$ such that $|L(v)|\ge 4$ for all $v\in V(G)$. Let us prove that
$G$ is $L$-colorable by induction on the number of vertices of $G$.
If $G$ is peculiar, then by Lemma~\ref{lem:colpec4} we know that the
theorem holds. So assume that $G$ is not peculiar. By
Theorem~\ref{thm:chvsbi} and Lemma~\ref{lem:peculiar}, we know that
$G$ can be decomposed by clique cutsets into elementary graphs. We
may assume that:
\begin{equation}\label{nosimp}
\mbox{$G$ has no simplicial vertex.}
\end{equation}
Suppose that $x$ is a simplicial vertex in $G$. By the induction
hypothesis, $G\setminus \{x\}$ admits an $L$-coloring $f$. Since $x$
is
simplicial, it has at most three neighbors. So $f$ can be extended to
$x$ by choosing in $L(x)$ a color not assigned by $f$ to its
neighbors. Thus (\ref{nosimp}) holds.

\medskip

By the discussion after the definition of a clique cutset (Section~1),
$G$ admits an extremal cutset $C$, i.e., a minimal clique cutset such
that for some component $A$ of $G\setminus C$ the induced subgraph
$G[A\cup C]$ is an atom (i.e., has no clique cutset).  Since $C$ is
minimal, every vertex $x$ of $C$ has a neighbor in every component of
$G\setminus C$ (for otherwise $C\setminus \{x\}$ would be a clique
cutset), and it follows that $G\setminus C$ has only two components
$A_1, A_2$ (for otherwise $x$ would be the center of a claw).  For
$i=1,2$ let $G_i=G[C\cup A_i]$.  Hence we may assume that $G_2$ is
elementary.

By the induction hypothesis, the graph $G[C\cup A_1]$ is
$4$-choosable, so it admits an $L$-coloring $f$.  We will show that we
can extend this coloring to $G$.

By Theorem~\ref{thm:mafree}, $G_2$ is obtained by augmenting the
line-graph ${\cal L}(H)$ of a bipartite graph $H$.  For each augment
$(X,Y)$ of $G_2$, select a pair of adjacent vertices such that one is
in $X$ and the other is in $Y$.  Also select all vertices of $G_2$
that are not in any augment.  It is easy to see that ${\cal L}(H)$ is
isomorphic to the subgraph of $G_2$ induced by the selected vertices.
Without loss it will be convenient to view ${\cal L}(H)$ as equal to
that induced subgraph.  We claim that:
\begin{equation}\label{cxy}
\longbox{If there is an augment $(X,Y)$ in $G_2$ such that both $C\cap
X$ and $C\cap Y$ are non-empty, then $V(G_2)=X\cup Y$.}
\end{equation}
Suppose on the contrary, under the hypothesis of (\ref{cxy}), that
$V(G_2)\neq X\cup Y$.  Let $Z=V(G_2)\setminus(X\cup Y)$.  Let
$Z_X=\{z\in Z \mid z$ has a neighbor in $X\}$ and $Z_Y=\{z\in Z \mid
z$ has a neighbor in $Y\}$.  By the definition of an augment, $Z_X$ is
complete to $X$ and anticomplete to $Y$, and $Z_Y$ is complete to $Y$
and anticomplete to $X$, and $Z_X\cap Z_Y=\emptyset$.  Since $G_2$ is
connected, we may assume up to symmetry that $Z_X\neq\emptyset$.  Pick
any $z\in Z_X$.  Since $G_2$ is an atom, $X$ is not a cutset of $G_2$
(separating $z$ from $Y$), so $Z_Y\neq\emptyset$, which restores the
symmetry between $X$ and $Y$.  Since $C$ is a clique and has a vertex
in $Y$, $C$ contains no vertex from $Z_X$; similarly, $C$ contains no
vertex from $Z_Y$; hence $C\subset X\cup Y$.  Pick any $x\in C\cap X$.
Since $C$ is a minimal cutset, $x$ has a neighor $a_1$ in $A_1$.  Then
$a_1$ must be adjacent to every neighbor $y$ of $x$ in $Y$, for
otherwise $\{x,a_1,z,y\}$ induces a claw; and it follows that $y\in
C$.  We can repeat this argument for every vertex in $C$; by the last
item in Theorem~\ref{thm:mafree} it follows that every vertex in
$X\cup Y$ is adjacent to $a_1$ and, consequently, is in $C$.  But this
is a contradiction because $C$ is a clique and $X\cup Y$ is not a
clique.  Thus (\ref{cxy}) holds.

\medskip

Now we distinguish two cases.

(I) First suppose that $G_2$ is not a cobipartite graph.

For every edge $uv$ in the bipartite multigraph $H$, let $C_{uv}$ be
the subset of $V(G_2)$ defined as follows. If $v$ has degree~$2$ in
$H$, say $N_H(v)=\{u,u'\}$, and $\{vu,vu'\}$ is a flat edge in ${\cal
L}(H)$
on which an augment $(X,X')$ of $G_2$ is based (where $X$ corresponds
to $vu$ and $X'$ corresponds to $vu'$), then let $C_{uv}=X$. If $uv$
is not such an edge, then let $C_{uv}$ be the set of parallel edges in
$H$ whose ends are $u$ and $v$. Now for every vertex $u$ in $H$, let
$C_u=\bigcup_{uv\in E(H)}C_{uv}$. Note that $C_u$ is a clique in
$G_2$. We claim that:
\begin{equation}\label{ccu}
\mbox{There is a vertex $u$ in $H$ such that $C=C_u$.}
\end{equation}
For every augment $(X,Y)$ in $G_2$ we have $V(G_2)\neq X\cup Y$,
because $G_2$ is not cobipartite, and so, by (\ref{cxy}), either
$C\cap X$ or $C\cap Y$ is empty.  It follows that there is a vertex
$u$ in $H$ such that $C\subseteq C_u$.  Suppose that $C\neq C_u$.
Then we can pick vertices $x\in C$ and $x'\in C_u\setminus C$ such
that $H$ has vertices $v,v'$ with $x\in C_{uv}$ and $x'\in C_{uv'}$.
Since $C$ is a minimal cutset, $x$ has a neighbor $a_1$ in $A_1$.
Since $G_2$ is an atom, the set $C_u\setminus C_{uv}$ is not a cutset,
so $x$ has a neighbor $z$ in $V(G_2)\setminus C_u$.  Then
$\{x,a_1,x',z\}$ induces a claw, a contradiction.  So $C= C_u$ and
(\ref{ccu}) holds.

\medskip

By (\ref{ccu}), let $u$ be a vertex in $H$ such that $C=C_u$.  Let
$D=\{d\in A_1\mid d$ has a neighbor in $C\}$.  We claim that:
\begin{equation}\label{dc}
\mbox{$D\cup C$ is a clique.}
\end{equation}
Pick any $d$ in $D$.  First suppose that $d$ is not complete to $C$.
Then we can find vertices $x\in C\cap N(d)$ and $x'\in C\setminus
N(d)$ such that $H$ has vertices $v,v'$ with $x\in C_{uv}$ and $x'\in
C_{uv'}$.  Since $G_2$ is an atom, the set $C_u\setminus C_{uv}$ is
not a cutset, so $x$ has a neighbor $z$ in $V(G_2)\setminus C_u$.
Then $\{x,d,x',z\}$ induces a claw, a contradiction.  It follows that
$D$ is complete to $C$.  Now suppose that $D$ contains non-adjacent
vertices $d,d'$.  Pick any $x\in C$.  Then $x$ has a neighbor $z$ in
$V(G_2)\setminus C_u$.  Then $\{x,d,d',z\}$ induces a claw, a
contradiction.  So $D$ is a clique.  Thus (\ref{dc}) holds.

\begin{equation}\label{dca2}
\mbox{$G[D\cup C\cup A_2]$ is an elementary graph.}
\end{equation}
Let $H^*$ be the bipartite graph obtained from $H$ by adding $|D|$
vertices of degree~$1$ adjacent to vertex $u$. Then it is easy to see
(by (\ref{ccu}) and (\ref{dc})) that $G[D\cup C\cup A_2]$ can be
obtained from ${\cal L}(H^*)$ by augmenting the same flat edges as
for $G_2$
and with the same augments. Thus (\ref{dca2}) holds.

\medskip

Let $D=\{d_1, \ldots, d_p\}$. (Actually we have $|C|\ge 2$ by
(\ref{ccu}) and consequently $|D|\le 2$ by (\ref{dc}), but we will not
use this fact.) Recall that $f$ is an $L$-coloring of $G_1$; so for
$i=1,\ldots,p$ let $c_i=f(d_i)$.

The maximum degree in $H^*$ is $\Delta(H^*)=\omega({\cal L}(H^*))\le
\omega(G_2)\le \omega(G)\le 4$.  So we can color the edges of $H^*$
with $4$ colors in such a way that vertices $d_1, \ldots, d_p$ receive
colors $c_1, \ldots, c_p$ respectively.  Let $L^*$ be a list
assignment on ${\cal L}(H^*)$ defined as follows.  If $v\in V({\cal
L}(H))$, let $L^*(v)=L(v)$.  For $i=1,\ldots,p$, let $L^*(d_i)=\{c_1,
\ldots, c_i\}$.  By Theorem~\ref{thm:galvin}, ${\cal L}(H^*)$ admits
an $L^*$-coloring $f^*$.  Now we can use the same technique as in the
proof of Theorem~\ref{thm-elementary-omega4} to extend $f^*$ to an
$L$-coloring of $G_2$.  Moreover, we have $f^*(d_1)=c_1$ and
consequently $f^*(d_i)=c_i=f(d_i)$ for all $i=1,\ldots, p$.  Let $f'$
be defined as follows.  For all $v\in V(G_1)\setminus C$, let
$f'(v)=f(v)$, and for all $v\in V(G_2)$, let $f'(v)=f^*(v)$.  Then
$f'$ is an $L$-coloring of $G$.  This completes the proof in case (I).

\bigskip

(II) We may now assume that $G_2$ is a cobipartite graph.  Let $W$ be
the set of vertices of $A_1$ that have a neighbor in $C$.  For all
$x\in C$, let $N_1(x)=N(x)\cap A_1$, $N_2(x)=N(x)\cap A_2$ and
$M_2(x)=A_2\setminus N(x)$.  We observe that:
\begin{equation}\label{n12m2}
\longbox{$N_1(x)$ and $N_2(x)$ are non-empty cliques, and $M_2(x)$ is
a clique.}
\end{equation}
We know that $N_1(x)$ and $N_2(x)$ are non-empty because $C$ is a
minimal cutset.  For $i=1,2$ pick any $n_i\in N_i(x)$; then $N_i(x)$
is a clique, for otherwise $x$ is the center of a claw with $n_{3-i}$
and two non-adjacent vertices from $N_i(x)$.  Also $M_2(x)$ is a
clique, for otherwise $G_2$ contains a stable set of size~$3$.  Thus
(\ref{n12m2}) holds.

\medskip

Suppose that $|C|=1$.  Let $C=\{x\}$.  Then $M_2(x)$ is empty, for
otherwise $N_2(x)$ is a clique cutset in $G_2$ (separating $x$ from
$M_2(x)$).  So $G_2$ is a clique.  Then every vertex in $A_2$ is
simplicial, a contradiction to (\ref{nosimp}).  So $|C|\ge 2$.

\medskip

Suppose that two vertices $x$ and $y$ of $C$ have inclusionwise
incomparable neighborhoods in $A_1$.  So there is a vertex $a$ in
$A_1$ adjacent to $x$ and not to $y$, and there is a vertex $b$ in
$A_1$ adjacent to $y$ and not to $x$.  If a vertex $u$ in $A_2$ is
adjacent to $x$, then it is adjacent to $y$, for otherwise
$\{x,a,y,u\}$ induces a claw, and vice-versa.  So $N_2(x)=N_2(y)$, and
$|N_2(x)|\le 2$ (because $N_2(x)\cup\{x,y\}$ is a clique), and
$M_2(x)=M_2(y)$.  Suppose that $M_2(x)\neq\emptyset$.  Let $C'=\{u\in
C\setminus\{x,y\}\mid u$ is complete to $N_2(x)\}$.  Since $C'\cup
N_2(x)$ is a clique, it cannot be a cutset of $G_2$, so some vertex
$z$ in $C\setminus(C'\cup\{x,y\})$ has a neighbor $v$ in $M_2(x)$.
Since $z\notin C'$, $z$ has a non-neighbor $u$ in $N_2(x)$.  Then $za$
is an edge, for otherwise $\{x,a,z,u\}$ induces a claw.  But then
$\{z,a,y,v\}$ induces a claw, a contradiction.  So $M_2(x)=\emptyset$.
Thus $A_2=N_2(x)=N_2(y)$.  If the vertices in $A_2$ have pairwise
comparable neighborhoods in $C$, then it follows easily that the
vertex in $A_2$ with the smallest degree is simplicial in $G$, a
contradiction to (\ref{nosimp}).  So there are two vertices $u,v$ in
$A_2$ and two vertices $z,t$ in $C$ such that $tu,zv$ are edges and
$tv,zu$ are not edges.  Clearly $z,t\notin\{x,y\}$, so $|C|=4$.  Then
$za$ is an edge, for otherwise $\{x,a,z,u\}$ induces a claw; and
similarly, $zb,ta,tb$ are edges.  Then $ab$ is an edge, for otherwise
$\{z,a,b,v\}$ induces a claw.  Recall that since $G$ is perfect and
claw-free, the neighborhood of every vertex can be partitioned into
two cliques, and consequently (since $\omega(G)\le 4$) every vertex
has degree at most~$6$.  Hence $N(x)=\{y,z,t,a,u,v\}$ (because we
already know that $x$ is adjacent to these six vertices), and
similarly $N(y)=\{x,z,t,b,u,v\}$, $N(z)=\{x,y,t,a,b,v\}$, and
$N(t)=\{x,y,z,a,b,u\}$.  It follows that $A_2=\{u,v\}$ and
$W=\{a,b\}$.  Here we view $f$ as an $L$-coloring of $G_1\setminus
(C\cup\{a,b\})$ rather than of $G_1$, and we try to extend it to
$\{a,b\}\cup C\cup A_2$.  Let $S=\{s\in
V(G_1)\setminus(C\cup\{a,b\})\mid s$ has a neighbor in $\{a,b\}\}$.
If a vertex $s\in S$ is adjacent to $a$ and not to $b$, then
$\{a,s,b,x\}$ induces a claw, a contradiction.  By symmetry this
implies that $S$ is complete to $\{a,b\}$.  Then $S$ is a clique, for
otherwise $\{a,s,s',x\}$ induces a claw from some non-adjacent
$s,s'\in S$.  So $S\cup\{a,b\}$ is a clique, and so $|S|\le 2$.  We
remove the colors of $f(S)$ from the lists of $a$ and $b$.  By
Lemma~\ref{lem:c4} we can color the vertices of $W \cup C\cup \{u,v\}$
with colors from the lists thus reduced.  So $G$ is $L$-colorable.

\medskip

Therefore we may assume that any two vertices of $C$ have
inclusionwise comparable neighborhoods in $A_1$.  This implies that
some vertex $a_1$ in $A_1$ is complete to $C$, and that some vertex
$x$ in $C$ is complete to $W$.  Since $\{a_1\}\cup C$ is a clique, we
have $|C|\le 3$.  We have $W=N_1(x)$ and, by (\ref{n12m2}), $W$ is a
clique, so $|W|\le 3$.  Here we view $f$ as an $L$-coloring of
$G_1\setminus C$ rather than of $G_1$, and we try to extend it to
$C\cup A_2$.  If $|W|=1$ (i.e., $W=\{a_1\}$), we remove the color
$f(a_1)$ from the list of the vertices in $C$.  Then $G_2$ is a
cobipartite graph which, with the reduced lists, satisfies the
hypothesis of Lemma~\ref{lem:c2} or~\ref{lem:c3}, so $f$ can be
extended to $G_2$.  Hence assume that $|W|\ge 2$.

\medskip

Suppose that $W$ is complete to $C$. Then $W\cup C$ is a clique, so
$|W|= 2$ and $|C|=2$. Let $C=\{x,y\}$. Let $X=N_2(x)$, $Y=N_2(y)$,
and $Z= A_2\setminus (X\cup Y)$. Suppose that $Z\neq\emptyset$. By
(\ref{n12m2}) $Z\cup (X\setminus Y)$ is a clique, since it is a subset
of $M_2(y)$. Likewise, $Z\cup (Y\setminus X)$ is a clique. Moreover
$X\setminus Y$ is complete to $Y\setminus X$, for otherwise
$\{x,y,v,z,u\}$ induces a $C_5$ for some non-adjacent $u\in X\setminus
Y$ and $v\in Y\setminus X$ and for any $z\in Z$. It follows that
$X\cup Y$ is a clique cutset in $G_2$ (separating $\{x,y\}$ from $Z$),
a contradiction. So $Z=\emptyset$, and $A_2=X\cup Y$. Here we view
$f$ as an $L$-coloring of $G_1\setminus C$ rather than of $G_1$, and
we try to extend it to $C\cup A_2$. We remove the colors of $f(W)$
from the list of $x$ and $y$. Since $|W|= 2$, each of these lists
loses at most two colors. By Lemma~\ref{lem:c2} we can color the
vertices of $C\cup A_2$ with colors from the lists thus reduced. So
$G$ is $L$-colorable.

\medskip

Now assume that $W$ is not complete to $C$.  So some vertex $a_2$ in
$W$ has a non-neighbor $y$ in $C$.  Then $N_2(x)\cup\{y\}$ is a
clique, for otherwise $\{x,a_2,u,v\}$ induces a clique for any two
non-adjacent vertices $u,v\in X\cup\{y\}$.  Suppose that $M_2(x)$ is
empty.  So $A_2=N_2(x)$.  Then the vertices in $A_2$ have comparable
neighborhoods in $C$ (because they are complete to $\{x,y\}$ and
$|C|\le 3$), so the vertex in $A_2$ with the smallest degree is
simplicial, a contradiction to (\ref{nosimp}).  Therefore $M_2(x)$ is
not empty.  Since the clique $\{y\}\cup N_2(x)$ is not a cutset in
$G_2$, some vertex $z$ in $C\setminus\{x,y\}$ has a neighbor $v$ in
$M_2(x)$.  Hence $|C|=3$.  Then $z$ has a non-neighbor $u$ in
$N_2(x)$, for otherwise $\{y,z\}\cup N_2(x)$ is a clique cutset in
$G_2$ (separating $x$ from $v$).  Then $za_2$ is an edge, for
otherwise $\{x,a_2,z,u\}$ induces a claw; and $yv$ is an edge, for
otherwise $\{z,a_2,y,v\}$ induces a claw; and $uv$ is an edge since
$N_2(y)$ is a clique.  Moreover, if $N_2(x)$ contains a vertex $u'$
adjacent to $z$, then $vu'$ is an edge since $N_2(z)$ is a clique.
Since this holds for every vertex in $M_2(x)\cap N(z)$, we deduce that
$(M_2(x)\cap N(z))\cup\{y\}\cup N_2(x)$ is a clique $Q$.  If $v'$ is
any non-neighbor of $z$ in $M_2(x)$, then $Q$ is a clique cutset in
$G_2$ (separating $\{x,z\}$ from $v'$), a contradiction.  So
$M_2(x)\subset N(z)$.  Suppose that $|W|=3$.  Pick $a_3\in
W\setminus\{a_1,a_2\}$.  Then $a_3z$ is not an edge, for otherwise
$W\cup\{x,z\}$ is a clique of size~$5$.  So, by the same argument as
for $a_2$, we deduce that $a_3y$ is an edge.  But this means that $y$
and $z$ have inclusionwise incomparable neighborhoods in $A_1$
(because of $a_2,a_3$), a contradiction.  So $|W|=2$.  We remove the
color $f(a_1)$ from the lists of $x,y,z$ and remove the color $f(a_2)$
from the list of $x$ and $z$.  By Lemma~\ref{lem:c3b} we can color the
vertices of $C\cup A_2$ with colors from the lists thus reduced.  So
$G$ is $L$-colorable.  This completes the proof of the theorem.
\end{proof}

\clearpage
\bibliographystyle{plain}

\end{document}